\documentclass[11pt]{amsart}

\usepackage{amsmath,amssymb,amsthm}
\usepackage{upgreek}
\usepackage{amsaddr}
\usepackage{mathtools}
\usepackage{geometry}
\usepackage{graphicx}
\usepackage{booktabs}
\usepackage{microtype}
\usepackage{hyperref}
\usepackage{enumerate}
\usepackage{xcolor}
\usepackage{listings}
\usepackage{tikz}

\newcommand{\ShuffleSquares}{%
  \mathord{%
    \mathchoice
      {\ShuffleSquaresDrawing{1.8ex}}
      {\ShuffleSquaresDrawing{1.8ex}}
      {\ShuffleSquaresDrawing{1.3ex}}
      {\ShuffleSquaresDrawing{1.0ex}}%
  }%
}

\newcommand{\ShuffleSquaresDrawing}[1]{%
  \vcenter{\hbox{%
    \begin{tikzpicture}[x=#1,y=#1]
      \fill[white] (0,0) rectangle (1,1);

      \foreach \x/\y in {
        0/3, 2/3,
        1/2, 2/2,
        0/1, 3/1,
        1/0, 3/0
      }{
        \fill[black]
          ({\x/4},{\y/4})
          rectangle
          ({(\x+1)/4},{(\y+1)/4});
      }

      \draw[black,line width=0.12pt,step=0.25]
        (0,0) grid (1,1);

      \draw[black,line width=0.3pt]
        (0,0) rectangle (1,1);
    \end{tikzpicture}%
  }}%
}

\definecolor{ProofSquarePurple}{RGB}{112,48,160}
\newif\ifPurpleProofSquares
\PurpleProofSquarestrue 
\ifPurpleProofSquares
  \colorlet{ProofSquareFill}{ProofSquarePurple}
\else
  \colorlet{ProofSquareFill}{black}
\fi
\newcommand{\ProofSquareCellSize}{0.4ex}
\newcounter{proofsquare}
\AddToHook{env/proof/begin}{\stepcounter{proofsquare}}
\newcommand{\ProgressiveProofSquare}{%
  \begin{tikzpicture}[x=\ProofSquareCellSize,y=\ProofSquareCellSize,baseline=0pt]
    \fill[white] (0,0) rectangle (6,6);
    \pgfmathtruncatemacro{\filledcells}{min(max(\value{proofsquare},0),36)}
    \ifnum\filledcells>0\relax
      \foreach \cell in {1,...,\filledcells}{%
        \pgfmathtruncatemacro{\cellcolumn}{mod(\cell-1,6)}
        \pgfmathtruncatemacro{\cellrow}{5-floor((\cell-1)/6)}
        \fill[ProofSquareFill]
          (\cellcolumn,\cellrow) rectangle ++(1,1);
      }%
    \fi
    \draw[black,line width=0.1pt,step=1] (0,0) grid (6,6);
    \draw[black,line width=0.3pt] (0,0) rectangle (6,6);
  \end{tikzpicture}%
}

\graphicspath{{figures/}}
\hypersetup{
  colorlinks=true,
  linkcolor=blue!45!black,
  citecolor=green!35!black,
  urlcolor=blue!55!black
}

\newtheorem{theorem}{Theorem}[section]
\newtheorem{proposition}[theorem]{Proposition}
\newtheorem{lemma}[theorem]{Lemma}
\newtheorem{corollary}[theorem]{Corollary}
\theoremstyle{definition}
\newtheorem{definition}[theorem]{Definition}
\newtheorem{example}[theorem]{Example}

\theoremstyle{remark}
\newtheorem{remark}[theorem]{Remark}

\newcommand{\eps}{\varepsilon}
\newcommand{\cA}{\mathcal{A}}
\newcommand{\cB}{\mathcal{B}}
\newcommand{\cC}{\mathcal{C}}
\newcommand{\cQ}{\mathcal{Q}}

\newcommand{\K}{\textcolor{purple!69!black}{\boldsymbol{\mathsf{K}}}}
\newcommand{\EPK}{\operatorname{EPK}}
\newcommand{\SSq}{\ShuffleSquares}
\newcommand{\rev}{\operatorname{rev}}
\newcommand{\doi}[1]{\href{https://doi.org/#1}{\nolinkurl{doi:#1}}}
\newcommand{\word}[1]{\textcolor{purple!69!black}{\mathtt{#1}}}
\newcommand{\classcode}{\texttt{verify\_c1\_c2\_c3.py}}
\newcommand{\repaircode}{\texttt{verify\_local\_repairs.py}}

\lstdefinestyle{pseudocodestyle}{
  basicstyle=\ttfamily\small,
  numbers=none,
  frame=single,
  rulecolor=\color{black!25},
  backgroundcolor=\color{black!2},
  breaklines=true,
  breakatwhitespace=false,
  showstringspaces=false,
  keepspaces=true,
  columns=fullflexible,
  tabsize=2,
  captionpos=b,
  escapeinside={(*@}{@*)}
}

\title[Shuffle Squares in Differentiable Words]
{Shuffle Squares in Differentiable Words}

\author{Micha\l{} Zwierzy\'nski}
\email{Michal.Zwierzynski@pw.edu.pl}
\email{ORCID: 0000-0002-9627-1563}
\address{Warsaw University of~Technology\\
Faculty of~Mathematics and Information Science\\
ul. Koszykowa 75\\
00-662 Warsaw, Poland\\[0.5ex]
}


\begin{document}
\raggedbottom

\begin{abstract}
Experiments on~binary run-length differentiability lead to~sharp computer-assisted criteria for~shuffle squares. Every $C^3$-word of~length greater than~$16$ is a~shuffle square exactly when both letter multiplicities are even. All~$34$ nonempty even-Parikh exceptions are smooth and persist in~every higher differentiability class. For~$C^2$ the~sharp threshold is~$48$, with~$212$ nonempty exceptions. At~level~$C^1$ no global parity threshold exists, but every even-Parikh non-shuffle-square of~length at~least~$36$ has proper nonempty shuffle-square prefixes and suffixes. Exactly~$230$ nonempty even-Parikh $C^1$-words have no nonempty shuffle-square prefix. The~full tree avoiding such prefixes eventually consists of~$422$ periodic rays. Consequently, a~nonempty Kolakoski prefix is a~shuffle square exactly when both multiplicities are even, apart from~lengths~$4$ and~$8$. We also characterize classes of~morphisms reflecting shuffle squares. For~doubly binary words, we determine the~exact deletion distance and largest twins, and prove sharp bounds for~single local repairs. Exact recurrences, residual-state checks, and separate Python programs make the~finite computations reproducible.
\end{abstract}

\maketitle

\tableofcontents

\section{Introduction}

\noindent
A~word is a~\emph{shuffle square} if its positions can be divided into~two sets which, read from~left to~right, induce two identical subwords.  Shuffle squares were introduced by~Henshall, Rampersad, and Shallit \cite{HenshallRampersadShallit}.  Recent work includes \cite{DatkoPawlik,GrytczukPawlikRucinskiProceedings}.

Recognizing shuffle squares is NP-complete even over a~binary alphabet~\cite{BulteauVialette}. He, Huang, Nam, and Thaper~\cite{HeHuangNamThaper} obtained enumerative results and conjectured that almost all binary words with even Parikh vector are shuffle squares. He and Post~\cite{HePost} proved this conjecture. Our results give sharp criteria valid for~every sufficiently long word in~specified finite differentiability classes.

The~classical \textit{Kolakoski word} is the~self-generating binary word equal to~its own sequence of~run lengths.  It was already studied by~Oldenburger in~1939 \cite{Oldenburger} and was later posed as~Problem~5304 by~Kolakoski \cite{Kolakoski}.  For~this reason the~name \emph{Oldenburger--Kolakoski word} is also used.  We take the~version beginning with~$\word{1}$:
\begin{align}\label{eq:K}
  \K=\word{122112122122112112212112122112112}\cdots,
  \qquad \Delta(\K)=\K,
\end{align}
where $\Delta$ denotes run-length encoding.  See \cite{Sing} for~a~survey and \cite{OEIS000002} for~extensive data and references.

The~work of~Grytczuk, Pawlik, and Ruci\'nski develops structural and algorithmic aspects of~shuffle squares \cite{GrytczukPawlikRucinskiProceedings,GrytczukPawlikRucinskiPreprint}. In~its~expanded version, the~authors study prefixes of~the~Kolakoski word and formulate Problem~1 concerning the~characterization of~shuffle squares among doubly binary words \cite[Section~6.1]{GrytczukPawlikRucinskiPreprint}. This~problem motivated the~present investigation. Although we~do not resolve it~in~full, we~obtain partial answers for~several natural subclasses and determine exactly which prefixes of~the~Kolakoski word are shuffle squares.

The~starting point was an~experimental comparison between the~parities of~the~letter counts and exact shuffle-square recognition for~finite Kolakoski prefixes. The~short words at~lengths~$4$ and~$8$ are obstructions, while the~first shuffle-square prefix has length~$16$. The~subsequent agreement with~the~parity condition suggested that these two obstructions might be the~only ones. To~understand why, we moved from~individual prefixes to~exhaustive searches in~the~larger languages of~smooth and finitely differentiable words. These searches make it possible to~study a~local obstruction at~every position and then use it in~a~proof for~arbitrary lengths.

For~$m\geqslant 1$, write
\begin{align*}
  P_m=\K_1\K_2\cdots\K_m.
\end{align*}
For~a~finite binary word~$w$ and $\word{a}\in\{\word{1},\word{2}\}$, let $|w|_{\word{a}}$ denote the~number of~occurrences of~$\word{a}$ in~$w$.  The~pair $(|w|_{\word{1}},|w|_{\word{2}})$ is the~\emph{Parikh vector} of~$w$, in~the~standard terminology of~combinatorics on~words \cite{Lothaire}.  We call this vector \emph{even} when both coordinates are even.

We use the~mnemonic notation
\begin{align}\label{eq:EPKdef}
  \EPK&=\{m\geqslant 1: |P_m|_{\word{1}}\equiv |P_m|_{\word{2}}\equiv0\pmod 2\}.
\end{align}
Thus~$\EPK$ is the~set of~\emph{even-Parikh Kolakoski lengths}.  Its increasing enumeration is OEIS A398324 \cite{OEIS398324}:
\begin{align}\label{eq:EPKterms}
  4,8,16,20,24,28,34,36,40,44,48,56,60,64,66,68,72,76,80,\ldots.
\end{align}
One consequence of~our general finite-differentiability theorem is the~following exact characterization.

\begin{theorem}\label{thm:main-intro}
For~every $m\geqslant1$,
\begin{align*}
  P_m\text{ is a~shuffle square}
  \quad\Longleftrightarrow\quad
  m\in\EPK\setminus\{4,8\}.
\end{align*}
\end{theorem}

The~proof is given in~Section~\ref{sec:kolakoski-prefixes}, where we deduce the~characterization from~Theorem~\ref{thm:C3-parity} below.

Finite differentiability provides a~natural hierarchy in~which to~place a~stronger statement. Smooth words, also called $\cC^\infty$-words, form the~smallest class considered here. The~notion originates in~Dekking's study of~the~factor language of~the~Kolakoski word (see \cite{BertheBrlekChoquette,Carpi,CassaigneHenry,Dekking,Weakley}). The~argument combines elementary closure and factorization properties with~finite computations specified by~exact recurrences and pseudocode. The~accompanying programs reproduce these computations without external data.

\begin{theorem}\label{thm:C3-parity}
Let $w\in\cC^3$ and $|w|>16$. Then $w$ is a~shuffle square if and only if
\begin{align*}
  |w|_{\word{1}}\equiv |w|_{\word{2}}\equiv0\pmod2.
\end{align*}
The~bound~$16$ is sharp. Moreover, exactly~$34$ nonempty even-Parikh $C^3$-words are not shuffle squares, and all of~them are smooth.
\end{theorem}

The~proof is given in~Section~\ref{subsec:stable-regime}.

Since $\cC^\infty\subseteq\cC^k\subseteq\cC^3$ $(k\geqslant3)$, Theorem~\ref{thm:C3-parity} immediately applies to~every $\cC^k$-word with $k\geqslant3$, and in~particular to~every smooth word. The~set of~nonempty exceptions is the~same in~all these classes because all~$34$ $C^3$-exceptions are themselves smooth. The~next level behaves differently but still admits a~complete finite answer. By~Theorem~\ref{thm:C2-parity}, every $C^2$-word of~length greater than~$48$ satisfies the~same parity criterion, the~threshold~$48$ is sharp, and precisely~$212$ nonempty even-Parikh $C^2$-words fail to~be shuffle squares. At~the~$C^1$ level, the~odd-ABBA family gives arbitrarily long even-Parikh exceptions. Nevertheless, Theorem~\ref{thm:C1-reduction} below reduces every sufficiently long exception from either end.

Theorem~\ref{thm:C3-parity} applies to~every sufficiently long $C^3$ factor, not merely to~a~prefix of~$\K$. Its induction is driven by~the~universal length-$24$ lemma: every $C^3$-word of~length~$24$ has nonempty shuffle-square prefixes and suffixes. The~finite induction base consists of~all even-Parikh $C^3$-words of~the~twelve even lengths from~$18$ through~$40$.

The~length-$24$ lemma has additional local-to-global consequences which do not require an~even Parikh vector. It produces bounded-block greedy factorizations in~every class~$\cC^k$ with~$k\geqslant3$, and from every position in~any right-infinite word whose length-$24$ factors lie in~$C^3$. These statements isolate precisely which parts of~the~argument are specific to~$\K$: the~final identification and the~sharper gap bound for~its even-Parikh prefixes.

The~decisive experimental object is the~tree of~words having no nonempty shuffle-square prefix. At~levels~$C^3$ and~$C^2$, this tree terminates at~lengths~$24$ and~$40$, respectively. These finite observations suggest the~bounded local lemmas that drive the~proofs. At~level~$C^1$, the~tree instead settles at~$422$ surviving words per~level. Inspecting their continuations reveals the~periods~$\word{1122}$ and~$\word{2211}$. Mere stabilization of~the~counts would not prove eventual periodicity. The~residual loops and exits extracted from~these experiments form the~basis of~the~finite verification below. Thus each observed pattern leads to~a~specific mathematical step: a~local factorization, an~exhaustive extension reduction, or a~pumping argument.

Two further directions extend the~structural picture. Section~\ref{sec:morphisms} studies nonerasing morphisms which reflect shuffle squares. Odd powers and private symbols yield a~complete classification for~alphabets of~equal size. For~disjoint private pairs, the~ordered-graph method of~\cite{GrytczukPawlikRucinskiPreprint} gives an~exact three-run criterion and a~sufficient condition for~images with more runs. Section~\ref{sec:repairs} determines the~deletion distance and largest twins in~every doubly binary word. It also proves that each even-Parikh exception admits one adjacent transposition or equal-pair deletion near either endpoint, with sharp windows of~$10$ and~$11$ positions. A~transposition preserving~$C^1$ needs a~sharp window of~$15$ positions, while an~equal pair can be inserted before a~suffix of~length at~most~$6$. Four balanced appended letters always suffice for~completion.

All computer assistance is finite and explicit. The~$C^3$ and~$C^2$ bad-prefix trees used for~the~bounded local lemmas have~$175$ and~$823$ nodes, respectively. The~$C^1$ bad-prefix tree is infinite but becomes a~union of~$422$ verified periodic rays. Every finite search and pumping argument is specified by~a~displayed recurrence, exact counts, and pseudocode.

The~shuffle-square recognizers are proved exact. Complete executable implementations are supplied as~separate computational files, described in~Section~\ref{sec:computer-assistance}, while the~recurrences, pseudocode, decisive counts, and proofs remain in~the~article.

Kory Kilpatrick independently proved the~smooth-word parity theorem, including the~classification of~the~$34$ exceptions, and the~Kolakoski prefix characterization in~his preprint~\cite{Kilpatrick}. We learned of~this independent work through private correspondence with~Jaros\l{}aw Grytczuk and Bart\l{}omiej Pawlik. Kilpatrick's proof uses a~factorization into parity-return atoms and a~finite computational verification. Our proofs use trees of~words avoiding shuffle-square prefixes and residual transitions, and establish the~results above for~the~larger finite differentiability classes~$C^3$, $C^2$, and~$C^1$.

\section{Shuffle squares and an~exact residual-state test}

\noindent
Throughout, the~alphabet is $\cA=\{\word{1},\word{2}\}$.  If $w=\word{w}_1\cdots \word{w}_n\in\cA^*$, its length is $|w|=n$.  The~empty word is denoted by~$\eps$. Here, a~Parikh vector is called \textit{even} if each of~its~coordinates is even, or equivalently, if every letter occurs an~even number of~times. For~example, $\word{1221}$ has Parikh vector $(2,2)$, so it~is even, whereas $\word{122112}$ has Parikh vector $(3,3)$, whose coordinates are odd.

\begin{definition}[\cite{Lothaire}]
\label{def:factor}
Let $u,w\in\cA^*$.  The~word~$u$ is a~\emph{factor} of~$w$ if there are words $x,y\in\cA^*$ such that
\begin{align*}
  w=xuy.
\end{align*}
Thus a~factor is a~contiguous subword.  It is a~\emph{prefix} when one can take $x=\eps$, and a~\emph{suffix} when one can take $y=\eps$.  Empty factors, prefixes, and suffixes are allowed unless ``nonempty'' is stated. For~example,
\begin{align*}
  \word{122112}=(\word{12})(\word{211})(\word{2}),
\end{align*}
so $\word{211}$ is a~factor.  The~word $\word{122}$ is a~prefix, while $\word{112}$ is a~suffix.
\end{definition}

\begin{definition}[\cite{Lothaire}]
\label{def:reversal}
For~a~word $w=\word{w}_1\word{w}_2\cdots \word{w}_n$, its \emph{reversal} is the~word
\begin{align*}
  \rev(w):=\word{w}_n\word{w}_{n-1}\cdots \word{w}_1.
\end{align*}
Its \emph{binary complement} $\overline w$ is obtained by~interchanging $\word{1}$ and~$\word{2}$ in~every position.  We put $\rev(\eps)=\overline\eps=\eps$.  For~example,
\begin{align*}
  \rev(\word{12211})=\word{11221},
  \qquad \overline{\word{12211}}=\word{21122}.
\end{align*}
Directly from~the~definition, $\rev(\rev(w))=w$ and $\rev(uv)=\rev(v)\rev(u)$.
\end{definition}

For~a~set $I=\{i_1<\cdots<i_r\}\subseteq\{1,\ldots,n\}$, put
\begin{align*}
  w|_I=\word{w}_{i_1}\cdots \word{w}_{i_r}.
\end{align*}

The~next definition is the~standard positional formulation of~the~notion introduced by~Henshall, Rampersad, and Shallit \cite{HenshallRampersadShallit}.

\begin{definition}[\cite{HenshallRampersadShallit}]
\label{def:shuffle-square}
A~nonempty word $w\in\cA^*$ is a~\emph{shuffle square} if there is a~partition
\begin{align*}
  \{1,\ldots,|w|\}=I\mathbin{\dot\cup}J
\end{align*}
such that $w|_I=w|_J$.  The~common word is called a~\emph{shuffle-square root} of~$w$.  Either set $I$ or $J$ is called a~\emph{support} of~one copy of~the~root.
\end{definition}

The~terminology \emph{root} agrees with~the~usage in~the~recent literature (see, for~example, Datko and Pawlik \cite{DatkoPawlik}).

Equivalently, the~letters of~$w$ can be colored red and blue so that the~red subword and the~blue subword, read in~their original order, coincide.

\begin{example}\label{ex:first-decomposition}
Consider
\begin{align*}
  P_{16}=\word{1221121221221121}.
\end{align*}

Take $I=\{1,2,3,4,6,7,11,13\}$ and $J=\{5,8,9,10,12,14,15,16\}$. Then $P_{16}|_I=P_{16}|_J=\word{12212121}$. Thus $P_{16}$ is a~shuffle square.
\end{example}

We record the~following elementary lemma, whose proof is immediate from~the~definition.

\begin{lemma}
\label{lem:even-necessary}
If $w$ is a~shuffle square, then $|w|_{\word{a}}$ is even for~every letter~$\word{a}$. In~particular, every shuffle square has even length.
\end{lemma}

The~converse is false, even over a~binary alphabet.

\begin{example}\label{ex:1221}
The~word $\word{1221}$ contains two copies of~each letter but is not a~shuffle square.  A~root would have one~$\word{1}$ and one~$\word{2}$, so it would be either $\word{12}$ or $\word{21}$.  The~last~$\word{1}$ cannot belong to~a~copy of~$\word{12}$, and the~first~$\word{1}$ cannot belong to~a~copy of~$\word{21}$.
\end{example}

We shall repeatedly use the~following elementary closure properties.

\begin{lemma}\label{lem:closure-shuffle}
Let $u$ and $v$ be shuffle squares.
\begin{enumerate}[(i)]
\item Their concatenation $uv$ is a~shuffle square.
\item Their reversals $\rev(u)$ and $\rev(v)$ are shuffle squares.
\item Their binary complements $\overline u$ and $\overline v$ are shuffle squares.
\end{enumerate}
\end{lemma}

\begin{proof}
If $u$ is a~self-shuffle of~$x$ and $v$ is a~self-shuffle of~$y$, then the~two chosen copies in~$u$ followed by~the~corresponding two copies in~$v$ show that $uv$ is a~self-shuffle of~$xy$. Reversing a~coloring that decomposes $u$ into two copies of~$x$ gives a~coloring that decomposes $\rev(u)$ into two copies of~$\rev(x)$. Applying the~same letter permutation to~both copies gives a~coloring of~$\overline u$ by~two copies of~$\overline x$, and similarly for~$v$.
\end{proof}

\subsection{An~exact recognizer}

\noindent
We recall the~queue-automaton characterization of~shuffle squares from~\cite[Section~2]{BussSoltys}. The~sets of~residual words below are the~buffer sets of~He and Post, and recurrence~\eqref{eq:residual-recurrence2} is their Proposition~2.1~\cite{HePost}. The~queue interpretation also appears in~\cite[Lemma~2.4]{Kilpatrick}. We include a~proof and explicit pseudocode to~make the~subsequent finite verifications self-contained.

For~$w=\word{w}_1\cdots \word{w}_n$, define finite sets of~residual words recursively. The~initial state is
\begin{align*}
  \cQ_0(w)=\{\eps\},
\end{align*}
and, for~$0\leqslant i<n$,
\begin{align}\label{eq:residual-recurrence2}
  \cQ_{i+1}(w)
  ={}&\{q \word{w}_{i+1}:q\in\cQ_i(w)\}\cup
  \{r:q=\word{w}_{i+1}r\text{ for~some }q\in\cQ_i(w)\}.
\end{align}

The~intended interpretation is as~follows.  After the~first~$i$ letters have been assigned to~two prospective copies, the~two partial subwords must be prefix-comparable in~any assignment that can eventually succeed.  If the~longer one equals the~shorter one followed by~$q$, then~$q$ is the~residual word.  The~identity of~the~longer color is irrelevant by~symmetry.  The~next letter may be appended to~the~longer copy, changing $q$ to~$q\word{w}_{i+1}$, or it may be appended to~the~shorter copy, which is possible exactly when it matches the~first letter of~$q$.

\begin{proposition}\label{prop:residual-test}
A~nonempty word $w$ is a~shuffle square if and only if $\eps\in\cQ_{|w|}(w)$.
\end{proposition}

\begin{proof}
Suppose first that a~coloring of~the~first~$i$ positions has produced two prefix-comparable subwords.  Removing their common prefix leaves a~unique residual word~$q$, together with~the~information of~which color is longer. Appending $\word{w}_{i+1}$ to~the~longer color gives the~first transition in~\eqref{eq:residual-recurrence2}.  Appending it to~the~shorter color is possible precisely if it cancels the~first letter of~$q$, giving the~second transition.  Thus every path in~the~recurrence records an~actual partial coloring, and every prefix-comparable partial coloring gives such a~path.

In~any successful coloring, both partial subwords after~$i$ positions are prefixes of~the~same final root.  Hence they are necessarily prefix-comparable.  At~the~end they are equal exactly when the~residual word is empty.  This proves both implications.
\end{proof}

\begin{remark}\label{rem:residual-pruning}
Suppose $|w|=n$ and a~residual~$q$ is present after $i$ letters.  Every subsequent transition changes $|q|$ by~exactly one.  Therefore a~state with~$|q|>n-i$ cannot reach~$\eps$ in~the~remaining $n-i$ steps and may be discarded without changing whether the~final empty state is reachable.

All finite verifications below use precisely this logically lossless pruning together with~the~necessary parity test from~Lemma~\ref{lem:even-necessary}.
\end{remark}

For~clarity, the~complete decision procedure used in~all later finite searches is given below. Its update before the~pruning line is exactly recurrence~\eqref{eq:residual-recurrence2}. Throughout the~listings, \texttt{*} denotes concatenation, \texttt{epsilon} denotes the~empty word, and set operations discard repetitions. An~\texttt{assert} statement checks the~displayed condition and stops the~verification if it fails.

\begin{lstlisting}[style=pseudocodestyle,caption={The exact residual-state recognizer},label={lst:residual-test}]
procedure SHUFFLE-SQUARE(w)
  if w is empty or some letter multiplicity is odd then return false
  Q <- {epsilon}
  for i <- 1,...,length(w) do
    NEXT <- {q*(*@$\word{w}_{i}$@*) : q in Q}
    NEXT <- NEXT union {r : q = (*@$\word{w}_{i}$@*)*r for some q in Q}
    Q <- {q in NEXT : length(q) <= length(w)-i}
  return (epsilon is in Q)
\end{lstlisting}

Listing~\ref{lst:residual-test} is therefore a~literal implementation of~Proposition~\ref{prop:residual-test}, including the~lossless pruning from~Remark~\ref{rem:residual-pruning}.

For~the~larger $C^2$ enumeration it is convenient to~process a~whole run at~once. We record an~equivalent exact test closely related to~the~canonical-twins representation in~\cite{GrytczukPawlikRucinskiPreprint}. For a~word $w=\word{w}_1\cdots \word{w}_n$, a~pair of \emph{perfect twins} is a~pair of disjoint sets of~positions $I,J\subseteq\{1,\ldots,n\}$ such that $I\mathbin{\dot\cup}J=\{1,\ldots,n\}$ and, writing $I=\{i_1<\cdots<i_m\}$ and $J=\{j_1<\cdots<j_m\}$, one has
\begin{align*}
\word{w}_{i_1}\cdots \word{w}_{i_m}
=
\word{w}_{j_1}\cdots \word{w}_{j_m}.
\end{align*}
Thus, the~two induced subwords are equal and together use every position of~$w$. We call such a~pair \emph{canonical} if, within each run of~$w$, all positions assigned to~the~first copy precede all positions assigned to~the~second copy.

\begin{proposition}\label{prop:run-recognizer}
The~procedure in~Listing~\ref{lst:run-test} returns true exactly for~the~nonempty shuffle squares.
\end{proposition}

\begin{lstlisting}[style=pseudocodestyle,caption={The exact run-capacity recognizer},label={lst:run-test}]
procedure RUN-SHUFFLE-SQUARE(w)
  if w is empty or some letter multiplicity is odd then return false
  R <- {epsilon}
  for each run (*@$\word{a}$@*)^r of w, from left to right, do
    NEXT <- empty set
    for q in R and x <- 0,...,r do
      y <- r-x
      z <- q*(*@$\word{a}$@*)^x
      if z begins with (*@$\word{a}$@*)^y then insert z with that prefix removed into NEXT
    R <- {q in NEXT : length(q) <= number of unprocessed letters}
  return (epsilon is in R)
\end{lstlisting}

\begin{proof}
Every shuffle square admits a~canonical pair of~perfect twins with~$i_\ell<j_\ell$ for~every~$\ell$. Indeed, pair the~positions carrying corresponding letters of~the~two equal copies and replace each pair by~its left and right endpoint. Since the~two original support sequences are increasing, the~sequences of~left and right endpoints are also increasing, and both still induce the~same word. If within one run a~second-copy position precedes a~first-copy position, interchange their assignments. The~letters are equal, so neither induced word changes. Moreover, the~first-copy count at every initial cut can only increase. Repeating removes all such inversions.

Now scan the~canonical pair just constructed from~left to~right. The~partial second copy is always a~prefix of~the~partial first copy. Let~$q$ be the~unmatched suffix. A~run $\word{a}^r$ is split into $\word{a}^x$ assigned to~the~first copy followed by~$\word{a}^y$ assigned to~the~second, where $x+y=r$. After the~first assignment the~unmatched word is $q\word{a}^x$. The~second assignment is possible exactly when $q\word{a}^x$ begins with~$\word{a}^y$, and deleting this prefix gives the~new residual. This is precisely one transition of~Listing~\ref{lst:run-test}. A~residual longer than the~unprocessed part cannot be cancelled and is discarded without loss, exactly as~in~Remark~\ref{rem:residual-pruning}. Thus canonical perfect twins give a~path from~$\eps$ to~$\eps$. Conversely, every such path specifies the~splits of~all runs and maintains the~prefix invariant, hence constructs two equal subwords partitioning all positions.
\end{proof}

\begin{example}\label{ex:length8}
For~$w=P_8=\word{12211212}$, recurrence \eqref{eq:residual-recurrence2} gives
\begin{align*}
\cQ_0&=\{\eps\},\quad
\cQ_1=\{\word{1}\},\quad
\cQ_2=\{\word{12}\},\quad
\cQ_3=\{\word{122}\},\quad
\cQ_4=\{\word{22},\word{1221}\},\\
\cQ_5&=\{\word{221},\word{12211}\},\quad
\cQ_6=\{\word{21},\word{2212},\word{122112}\},\quad
\cQ_7=\{\word{211},\word{22112},\word{22121},\word{1221121}\},\\
\cQ_8&=\{\word{11},\word{2112},\word{2121},\word{221122},\word{221212},\word{12211212}\}.
\end{align*}
Since $\eps\notin\cQ_8$, the~word $P_8$ is not a~shuffle square. The~same computation at~length~$4$ gives $\cQ_4(\word{1221})=\{\word{22},\word{1221}\}$, in~agreement with~Example~\ref{ex:1221}.

For~comparison, the~positive example $v=\word{1212}$ gives
\begin{align*}
  \cQ_0(v)=\{\eps\},\quad
  \cQ_1(v)=\{\word{1}\},\quad
  \cQ_2(v)=\{\word{12}\},\quad
  \cQ_3(v)=\{\word{2},\word{121}\},\quad 
  \cQ_4(v)=\{\eps,\word{22},\word{1212}\}.
\end{align*}
Here $\eps\in\cQ_4(v)$.  Indeed, the~supports $\{1,2\}$ and $\{3,4\}$ both induce the~root~$\word{12}$.
\end{example}

\section{Run-length encoding and differentiability classes}

\noindent
The~finite derivative used in~this section is due to~Dekking \cite{Dekking}.  Its purpose is to~apply run-length encoding to~a~finite factor without pretending that its two endpoint runs are complete. The~notation $\cC^\infty$ for~words derivable arbitrarily many times was already standard in~the~work of~Weakley and Carpi \cite{Carpi,Weakley}.  Such words are now commonly called \emph{smooth words} \cite{BertheBrlekChoquette}.  Recent literature sometimes says \emph{finite smooth} or \emph{$f$-smooth} to~distinguish this notion from~smooth infinite words (see \cite[Section~1.2]{CassaigneHenry}).

Write a~nonempty binary word in~its run decomposition
\begin{align*}
  w=\word{a}_1^{r_1}\word{a}_2^{r_2}\cdots \word{a}_t^{r_t},
  \qquad \word{a}_i\neq \word{a}_{i+1},\quad r_i\geqslant1.
\end{align*}
Its run-length encoding is
\begin{align*}
  \Delta(w)=\word{r}_1\word{r}_2\cdots\word{r}_t.
\end{align*}
For~example,
\begin{align*}
  \word{122112}=\word{1}^1 \word{2}^2 \word{1}^2 \word{2}^1,
  \qquad \Delta(\word{122112})=\word{1221}.
\end{align*}

The~ordinary run-length encoding of~a~finite factor does not faithfully record a~run cut at~either endpoint.  The~finite derivative removes this boundary ambiguity.

\begin{definition}[\cite{Dekking}]
\label{def:derivative}
The~finite derivative $D(w)$ of~a~binary word~$w$ is defined as~follows. If one of~the~runs of~$w$ has length greater than~$2$, then $D(w)$ is undefined.  Otherwise start with~$\Delta(w)$, delete its first symbol if that symbol is~$\word{1}$, and then delete its last symbol if that symbol is~$\word{1}$. A~single symbol is deleted at~most once.  We also put $D(\eps)=\eps$.
\end{definition}

Thus
\begin{align*}
  D(\word{122112})=\word{22},
  \qquad D(\word{2211})=\word{22},
  \qquad D(\word{1})=D(\word{21})=\eps.
\end{align*}
On~the~other hand, $D(\word{12121})=\word{111}$, and the~next derivative is undefined because $\word{111}$ contains a~run of~length~$3$.

\begin{definition}
\label{def:Ck}
For~an~integer $k\geqslant0$, let
\begin{align*}
  \cC^k=\{w\in\cA^*:D^j(w)\text{ is defined for~every }1\leqslant j\leqslant k\}.
\end{align*}
We call a~member of~$\cC^k$ a~\emph{$\cC^k$-word}. Thus $\cC^0=\cA^*$, and the~convention $D(\eps)=\eps$ allows a~derivative tower which reaches~$\eps$ early to~be continued indefinitely.
\end{definition}

\begin{definition}[\cite{Dekking,Weakley}]
\label{def:smooth}
A~finite binary word~$w$ is \emph{smooth}, or a~$\cC^\infty$-word, if its finite derivatives remain defined until the~empty word is reached. Equivalently,
\begin{align*}
  w\in\cC^\infty:=\bigcap_{k\geqslant0}\cC^k.
\end{align*}
\end{definition}

The~first levels of~the~finite differentiability hierarchy are genuinely different. For~example,
\begin{align*}
  \word{11211211}\mapsto\word{21212}\mapsto\word{111},
\end{align*}
and the~next derivative is undefined. Hence $\word{11211211}\in\cC^2\setminus\cC^3$. Its complement $\word{22122122}$ gives a~second example.

A~complete enumeration of~the~binary words through length~$8$ shows that these are the~two shortest members of~$\cC^2\setminus\cC^3$.

After relabelling the~alphabet $\{\word{0},\word{1}\}$ as~$\{\word{1},\word{2}\}$, the~doubly binary words of~Grytczuk, Pawlik, and Ruci\'nski \cite{GrytczukPawlikRucinskiPreprint} are precisely the~$C^1$-words. Their central odd-ABBA family illustrates the~difference between $C^1$ and~$C^2$. In~their alphabet the~words are $(\word{1001})^n$. With our labels put
\begin{align*}
  W_n=(\word{2112})^n.
\end{align*}
Its run decomposition gives
\begin{align*}
  \Delta(W_n)=\word{1}\,\word{2}^{2n-1}\word{1},
  \qquad D(W_n)=\word{2}^{2n-1}.
\end{align*}
Consequently, $W_1$ is smooth, whereas $W_n\in\cC^1\setminus\cC^2$ for $n\geqslant 2$. In~particular, all odd-power members $W_n$ of~their infinite non-shuffle-square family lie outside~$\cC^2$ except for~$W_1$. Thus the~$C^2$ result below is a~complete solution of~their Problem~1 on a~proper subclass, but it does not settle the~problem for~all doubly binary words.

\begin{definition}[\cite{CassaigneHenry,Dekking}]
\label{def:infinite-smooth}
A~right-infinite word $X\in\cA^{\mathbb N}$ is \emph{smooth} if every iterate $\Delta^j(X)$, $j\geqslant0$, is a~right-infinite word over~$\cA$. This is the~infinite version of~smoothness. It should be distinguished from~the~finite-word notion in~Definition~\ref{def:smooth}.
\end{definition}

For~finite words, smoothness is a~finite condition because every nonempty defined derivative is shorter than its input. For~instance, the~first shuffle-square prefix in~Example~\ref{ex:first-decomposition} has derivative tower
\begin{align*}
\word{1221121221221121}
  \mapsto\word{221121221}
   \mapsto\word{22112}
  \mapsto\word{22}
   \mapsto\word{2}
   \mapsto\eps.
\end{align*}

We record carefully the~closure properties needed later. Factoriality is classical for~smooth words (see \cite[Proposition~2]{Weakley} and the~modern account \cite[Section~1.2.2]{CassaigneHenry}). The~same endpoint argument works at~every finite level. We include the~proof so that the~argument is self-contained for~the~precise convention in~Definition~\ref{def:derivative}.

\begin{lemma}[\cite{CassaigneHenry,Weakley}]
\label{lem:smooth-closure}
For~every $k\geqslant0$ the~following statements hold.
\begin{enumerate}[(i)]
\item Every factor of~a~$\cC^k$-word is a~$\cC^k$-word.
\item A~word~$w$ belongs to~$\cC^k$ if and only if its reversal $\rev(w)$ belongs to~$\cC^k$, and whenever $D(w)$ is defined,
\begin{align*}
  D(\rev(w))=\rev(D(w)).
\end{align*}
\item A~word~$w$ belongs to~$\cC^k$ if and only if its complement $\overline w$ belongs to~$\cC^k$.
\end{enumerate}
The~corresponding three assertions also hold for~smooth words.
\end{lemma}

\begin{proof}
Let $u$ be a~factor of~a~word~$w$ whose derivative is defined.  Every run of~$u$, except possibly its first and last runs, is a~complete run of~$w$. If an~endpoint run of~$u$ is only a~one-letter fragment of~a~run of~$w$, its run length~$1$ is deleted by~$D$.  If an~endpoint run survives in~$D(u)$, then it has length~$2$ and is therefore a~complete run of~$w$.  Moreover, a~one-letter endpoint run of~$w$ cannot survive as~an~interior run of~$u$. It follows that
\begin{align}\label{eq:factor-derivative}
  D(u)\text{ is a~factor of~}D(w),
\end{align}
with~the~empty word allowed.  Iterating \eqref{eq:factor-derivative} proves~(i).

Reversal reverses the~list of~run lengths, while the~two boundary deletions in~Definition~\ref{def:derivative} are symmetric. Hence $D(\rev(w))=\rev(D(w))$, and iteration proves~(ii). Complementation does not change any run length, so $D(\overline w)=D(w)$ whenever the~derivative is defined. This proves~(iii). Since every smooth word belongs to~$\cC^k$ for~every~$k$, and its derivative tower reaches~$\eps$, the~same arguments give the~final assertion.
\end{proof}

The~next factor-language inclusion was the~original motivation for~Dekking's finite derivative.  In~the~notation of~\cite[Section~1.2.3]{CassaigneHenry}, it is the~inclusion $\mathcal L(\mathcal C^\infty)\subseteq\mathcal C_f^\infty$, where $\mathcal L$ denotes the~finite-factor language, $\mathcal C^\infty$ the~class of~smooth infinite words, and $\mathcal C_f^\infty$ the~class of~finite smooth words.  We include the~boundary argument, so the~cited inclusion is not being used as~an~unproved black box.

\begin{proposition}[\cite{CassaigneHenry,Dekking}]
\label{prop:infinite-factors-smooth}
Every finite factor of~a~smooth right-infinite binary word is a~finite smooth word.
\end{proposition}

\begin{proof}
Let $u$ be a~finite factor of~a~smooth right-infinite word~$X$.  Its endpoint runs may be truncated pieces of~runs of~$X$, and Definition~\ref{def:derivative} deletes exactly the~ambiguous one-letter endpoint run lengths.  All surviving run lengths are complete.  Hence
\begin{align*}
  D(u)\text{ is a~factor of~}\Delta(X),
\end{align*}
with~the~empty word allowed.  Applying the~same argument successively gives
\begin{align*}
  D^j(u)\text{ is a~factor of~}\Delta^j(X)
\end{align*}
for~every~$j$ for~which $D^j(u)$ is nonempty.  Since all letters in~every $\Delta^j(X)$ belong to~$\{\word{1},\word{2}\}$, each derivative is defined.  A~nonempty finite derivative is strictly shorter than its input, so the~process reaches $\eps$.  Thus $u$ is finite smooth.
\end{proof}

Since $\Delta(\K)=\K$ by~\eqref{eq:K}, all iterates of~$\K$ are right-infinite words over~$\{\word{1},\word{2}\}$. Thus, $\K$ is smooth in~the~sense of~Definition~\ref{def:infinite-smooth}, and Proposition~\ref{prop:infinite-factors-smooth} immediately yields the~following classical result.

\begin{corollary}[\cite{CassaigneHenry,Dekking}]
\label{cor:K-factors-smooth}
Every finite factor of~the~Kolakoski word~$\K$ is smooth.
\end{corollary}

\section{Even prefixes of~the~Kolakoski word}

\noindent
Recall the~set~$\EPK$ from~\eqref{eq:EPKdef}.  Since $|P_m|_{\word{1}}+|P_m|_{\word{2}}=m$, the~following conditions characterize membership in~$\EPK$:
\begin{align*}
  m\equiv0\pmod2
  \quad\text{and}\quad
  |P_m|_{\word{1}}\equiv0\pmod2.
\end{align*}
We first give an~elementary strengthening of~the~infinitude assertion in~OEIS A398324 given by Pawlik.

\begin{proposition}\label{prop:even-gap}
The~set~$\EPK$ is infinite.  If $e<e'$ are consecutive elements of~$\EPK$, then
\begin{align*}
  e'-e\in\{2,4,6,8\}.
\end{align*}
\end{proposition}

\begin{proof}
Group the~letters of~$\K$ into~aligned pairs
\begin{align*}
  B_j=\K_{2j-1}\K_{2j}\qquad(j\geqslant1).
\end{align*}
When one passes from~$P_{2j-2}$ to~$P_{2j}$, the~parity of~the~number of~ones changes exactly when $B_j$ is mixed, that is, when $B_j\in\{\word{12},\word{21}\}$.  An~equal pair $\word{11}$ or $\word{22}$ leaves this parity unchanged.

Start at~an~element $e\in\EPK$.  If the~next aligned pair is equal, then $e+2\in\EPK$.  Otherwise the~next pair is mixed and the~parity becomes odd. There cannot be three consecutive equal aligned pairs.  Indeed, equal aligned pairs are full runs, since all runs of~$\K$ have length at~most~$2$. Consecutive equal pairs must alternate between $\word{11}$ and~$\word{22}$.  Three consecutive equal pairs would therefore give three consecutive runs of~length~$2$.  Their run lengths would produce the~factor $\word{222}$ in~$\Delta(\K)=\K$, impossible because every run of~$\K$ has length at~most~$2$.

Consequently, after a~mixed pair, another mixed pair must occur within the~next three aligned pairs.  At~that second mixed pair the~parity returns to~even.

Thus parity returns after at~most four aligned pairs, proving $e'-e\leqslant8$.  Since $4\in\EPK$, the~argument can be iterated indefinitely.
\end{proof}

\begin{example}\label{ex:even-prefixes}
The~successive differences of~the~initial terms in~\eqref{eq:EPKterms} are
\begin{align*}
  4,8,4,4,4,6,2,4,4,4,8,4,4,2,2,4,4,4,\ldots.
\end{align*}
\end{example}

\section{Shuffle squares in~finite differentiability classes}
\label{sec:Ck}

\noindent
For~$k\geqslant1$ and $n\geqslant0$, put $\cC_n^k=\cC^k\cap\cA^n$, and let $\cB_n^{(k)}$ be the~set of~$\cC^k$-words of~length~$n$ having no nonempty shuffle-square prefix. Factoriality gives the~exact recurrence
\begin{align}\label{eq:Ck-bad-tree}
 \cB_{n+1}^{(k)}
 =\{w\word{a}:w\in\cB_n^{(k)},\ \word{a}\in\{\word{1},\word{2}\},\
 w\word{a}\in\cC^k,\ w\word{a}\notin\SSq\},
 \qquad \cB_0^{(k)}=\{\eps\}.
\end{align}
Here and below, writing $u\in\SSq$ means that $u$ is a~nonempty shuffle square. Listing~\ref{lst:bad-prefix-tree} is exactly recurrence~\eqref{eq:Ck-bad-tree}.

\begin{lstlisting}[style=pseudocodestyle,caption={The pruned bad-prefix tree},label={lst:bad-prefix-tree}]
procedure BAD-PREFIX-TREE(k, L)
  B[0] <- {epsilon}
  for n <- 0,...,L-1 do
    B[n+1] <- empty set
    for w in B[n] and (*@$\word{a}$@*) in {(*@$\word{1}$@*),(*@$\word{2}$@*)} do
      if w*(*@$\word{a}$@*) is in C^k and w*(*@$\word{a}$@*) is not a shuffle square then
        insert w*(*@$\word{a}$@*) into B[n+1]
  return B[0],...,B[L]
\end{lstlisting}

For~the~unpruned level and parity checks we use Listing~\ref{lst:parity-levels}. The~test \texttt{SHUFFLE-SQUARE} is the~exact residual-state test in~Listing~\ref{lst:residual-test}. For~each length~$1\leqslant n\leqslant L$, the~output sets \texttt{C[n]}, \texttt{E[n]}, and~\texttt{N[n]} contain, respectively, all $\cC^k$-words, those with even Parikh vector, and those even-Parikh words that fail the~shuffle-square test.

\begin{lstlisting}[style=pseudocodestyle,caption={Complete level and parity enumeration},label={lst:parity-levels}]
procedure PARITY-LEVELS(k, L)
  C[0] <- {epsilon}
  for n <- 0,...,L-1 do
    C[n+1] <- {w*(*@$\word{a}$@*) : w in C[n], (*@$\word{a}$@*) in {(*@$\word{1}$@*),(*@$\word{2}$@*)}, w*(*@$\word{a}$@*) is in C^k}
    E[n+1] <- {w in C[n+1] : both multiplicities in w are even}
    N[n+1] <- {w in E[n+1] : SHUFFLE-SQUARE(w) is false}
  return C, E, N
\end{lstlisting}

\subsection{Differentiability of~order at~least three}\label{subsec:stable-regime}

The~local statement needed for~the~global induction already holds for~$\cC^3$.

\begin{lemma}\label{lem:C3-local}
Every $C^3$-word of~length~$24$ has a~nonempty shuffle-square prefix and a~nonempty shuffle-square suffix.
\end{lemma}

\begin{proof}
Listing~\ref{lst:bad-prefix-tree}, with $k=3$ and $L=24$, gives precisely
\begin{align*}
(|\cB_n^{(3)}|)_{n=0}^{24}
={}&(1,2,2,4,4,6,6,6,6,8,10,10,10,14,14,18,10,10,8,10,4,4,4,4,0).
\end{align*}
Thus the~complete tree has only~$175$ nodes and $\cB_{24}^{(3)}=\varnothing$, proving the~prefix assertion. The~suffix assertion follows by~reversal, using Lemmas~\ref{lem:smooth-closure}(ii) and~\ref{lem:closure-shuffle}(ii).

For~sharpness, put
\begin{align}\label{eq:A23}
 A=\word{12211211212212211211221}.
\end{align}
The~last nonempty level is
\begin{align}\label{eq:bad23}
 \cB_{23}^{(3)}
 =\{A,\rev(A),\overline A,\overline{\rev(A)}\}.
\end{align}
All four words are smooth, so the~local constant~$24$ cannot be decreased even in~the~smooth subclass.
\end{proof}

For~the~finite induction base, Listing~\ref{lst:parity-levels} gives Table~\ref{tab:C3-base}. The~third column counts all even-Parikh words. The~fourth records how many of~them the~exact recognizer rejects.

\begin{table}[ht]
\centering
\caption{The~complete finite base for~the~$C^3$ parity theorem}
\label{tab:C3-base}
\small
\begin{tabular}{@{}rrrr@{\qquad}rrrr@{}}
\toprule
$n$ & $|\cC_n^3|$ & even & rejected & $n$ & $|\cC_n^3|$ & even & rejected\\
\midrule
18 & 284   & 134   & 0 & 30 & 3792  & 1856  & 0\\
20 & 440   & 232   & 0 & 32 & 5838  & 2970  & 0\\
22 & 674   & 324   & 0 & 34 & 8998  & 4438  & 0\\
24 & 1046  & 538   & 0 & 36 & 13822 & 6990  & 0\\
26 & 1600  & 778   & 0 & 38 & 21332 & 10550 & 0\\
28 & 2470  & 1266  & 0 & 40 & 32740 & 16524 & 0\\
\midrule
\multicolumn{8}{c}{$\displaystyle\text{total number of~even-Parikh words tested}=46600$}\\
\bottomrule
\end{tabular}
\end{table}

\begin{proof}[Proof of~Theorem~\ref{thm:C3-parity}]
Necessity is Lemma~\ref{lem:even-necessary}. For~sufficiency, use strong induction over the~even lengths. The~twelve lengths $18,20,\ldots,40$ are the~complete base in~Table~\ref{tab:C3-base}. If $|w|\geqslant42$, the~length-$24$ suffix of~$w$ belongs to~$\cC^3$ by Lemma~\ref{lem:smooth-closure}(i). Lemma~\ref{lem:C3-local} gives it a~nonempty shuffle-square suffix~$U$ with $2\leqslant|U|\leqslant24$. Write $w=vU$. Both $w$ and~$U$ have even Parikh vector, hence so does~$v$. Moreover, $v\in\cC^3$ and $|v|\geqslant18$. The~induction hypothesis makes $v$ a~shuffle square, and Lemma~\ref{lem:closure-shuffle}(i) then makes $w=vU$ a~shuffle square.

For~sharpness, consider
\begin{align}\label{eq:C16}
 C=\word{1221121121221221}.
\end{align}
It has eight occurrences of~each letter and derivative tower
\begin{align*}
 C\mapsto\word{221211212}\mapsto\word{211211}
  \mapsto\word{212}\mapsto\word{1}\mapsto\eps,
\end{align*}
so it is smooth. The~exact residual test rejects~$C$. In~fact, $C\in\cB_{16}^{(3)}$. Hence the~threshold cannot be lowered.

Finally, the~complete calculation at~the~even lengths at~most~$16$ gives exactly~$34$ nonempty even-Parikh $C^3$-words rejected by~the~recognizer. Their distribution is
\begin{align}\label{eq:C3-exception-counts}
 (n,e_n)=(4,2),(8,6),(10,4),(12,10),(14,4),(16,8).
\end{align}
Direct iteration of~the~finite derivative verifies that every one of these~$34$ words is smooth.

Recurrence~\eqref{eq:Ck-bad-tree}, Listings~\ref{lst:bad-prefix-tree} and~\ref{lst:parity-levels}, the~full bad-tree vector, and Table~\ref{tab:C3-base} give a~finite verification for every exhaustive step used in~this proof.
\end{proof}

The~single $C^3$ proof contains the~smooth case and every intermediate finite differentiability class. For~smooth words, the~following criterion and exceptional set were also obtained independently in~\cite{Kilpatrick}.

\begin{corollary}\label{cor:stable-regime}
For~every $k\geqslant3$, an~even-Parikh $\cC^k$-word of~length greater than~$16$ is a~shuffle square. The~threshold~$16$ and the~set of exactly~$34$ nonempty exceptions are the~same for~all $k\geqslant3$ and for~smooth words.
\end{corollary}

\begin{corollary}\label{cor:factorwise-parity}
Let $X$ be either a~finite $\cC^k$-word for~some $k\geqslant3$, or a right-infinite smooth binary word. Every factor~$u$ of~$X$ with $|u|>16$ is a~shuffle square if and only if its Parikh vector is even.
\end{corollary}

\begin{proof}
In~the~finite case, Lemma~\ref{lem:smooth-closure}(i) makes~$u$ a $\cC^k$-word and hence a~$C^3$-word. In~the~right-infinite case, Proposition~\ref{prop:infinite-factors-smooth} makes~$u$ smooth. Apply Theorem~\ref{thm:C3-parity} and Lemma~\ref{lem:even-necessary}.
\end{proof}

\subsection{The~class \texorpdfstring{$C^2$}{C2}}

At~the~next level, the~local threshold increases from~$24$ to~$40$.

\begin{lemma}\label{lem:C2-local}
Every $C^2$-word of~length~$40$ has a~nonempty shuffle-square prefix and a~nonempty shuffle-square suffix. The~constant~$40$ is optimal.
\end{lemma}

\begin{proof}
Listing~\ref{lst:bad-prefix-tree}, with $k=2$ and $L=40$, gives
\begin{align*}
(|\cB_n^{(2)}|)_{n=0}^{40}
={}&(1,2,2,4,4,6,6,6,6,8,10,14,12,18,18,22,16,20,20,26,22,\\
&30,26,32,28,38,38,50,40,52,36,46,36,44,22,26,10,10,8,8,0).
\end{align*}
Thus the~complete tree has~$823$ nodes and $\cB_{40}^{(2)}=\varnothing$. This proves the~prefix assertion, and reversal gives the~suffix assertion. Finally,
\begin{align}\label{eq:C2-bad39}
 \word{122112112122122122112112112122122122112}=\word{122}(\word{112})^2(\word{122})^3(\word{112})^3(\word{122})^3\word{112}
\end{align}
belongs to~$\cB_{39}^{(2)}$, proving optimality.

The~displayed complete vector, together with recurrence~\eqref{eq:Ck-bad-tree} and Listing~\ref{lst:bad-prefix-tree}, is a~finite reproducible verification for this local lemma.
\end{proof}

The~local lemma also gives a~bounded factorization at the~$C^2$ level.

\begin{corollary}\label{cor:C2-factorization}
Every finite $C^2$-word~$w$ has a~factorization
\begin{align*}
 w=S_1S_2\cdots S_k r,
\end{align*}
obtained by~successively removing the~shortest nonempty shuffle-square prefix, where every $S_i$ is a~$C^2$ shuffle square satisfying $2\leqslant|S_i|\leqslant40$, and the~terminal remainder satisfies $r\in\cC^2$ and $|r|\leqslant39$. If $k\geqslant1$, then $S_1\cdots S_k$ is a~shuffle square. Both the~remainder bound~$39$ and the~block bound~$40$ are sharp.
\end{corollary}

\begin{proof}
Repeatedly remove the~shortest nonempty shuffle-square prefix. Every intermediate remainder and every removed block is a~$C^2$ factor by Lemma~\ref{lem:smooth-closure}(i). Lemma~\ref{lem:C2-local} bounds each chosen block by~$40$ and the~terminal remainder by~$39$. Concatenation gives the~last assertion about $S_1\cdots S_k$. The~word in
\eqref{eq:C2-bad39} proves sharpness of~the~remainder bound.
Table~\ref{tab:minimal-suffixes} contains eight minimal suffix squares of~length~$40$. Their reversals are minimal shuffle-square prefixes, proving sharpness of~the~block bound.
\end{proof}

To~obtain the~sharp global parity threshold, let~$\mathcal M$ be the set of~$C^2$ shuffle squares having no proper nonempty shuffle-square suffix. Lemma~\ref{lem:C2-local} implies
\begin{align}\label{eq:M-bound}
 |s|\leqslant40\qquad(s\in\mathcal M),
\end{align}
because the~last~$40$ letters of~a~longer member would contain a~proper shuffle-square suffix. A~$C^2$ extension of~a~word in $\cB_{n-1}^{(2)}$ is a~minimal shuffle-square prefix exactly when the extension itself is a~shuffle square. Reversal converts minimal prefixes into minimal suffixes. This enumerates~$\mathcal M$ and gives $|\mathcal M|=288$, with the~length distribution in Table~\ref{tab:minimal-suffixes}.

\begin{table}[ht]
\centering
\caption{Minimal $C^2$ shuffle-square suffixes}
\label{tab:minimal-suffixes}
\small
\begin{tabular}{@{}rr@{\qquad}rr@{}}
\toprule
length & number & length & number\\
\midrule
 2 & 2  & 22 & 16\\
 4 & 2  & 24 & 18\\
 6 & 4  & 26 & 12\\
 8 & 2  & 28 & 32\\
10 & 2  & 30 & 36\\
12 & 6  & 32 & 30\\
14 & 10 & 34 & 36\\
16 & 14 & 36 & 26\\
18 & 8  & 38 & 10\\
20 & 14 & 40 & 8\\
\midrule
\multicolumn{4}{c}{total: $288$}\\
\bottomrule
\end{tabular}
\end{table}

\begin{lemma}\label{lem:C2-reduction}
Let $w\in\cC^2$ have even Parikh vector, suppose that $w$ is not a shuffle square, and assume $|w|\geqslant40$. Then $w=us$ for~some $s\in\mathcal M$, where $u\in\cC^2$ also has even Parikh vector and is not a~shuffle square.
\end{lemma}

\begin{proof}
By~Lemma~\ref{lem:C2-local}, the~length-$40$ suffix factor of~$w$ has a nonempty shuffle-square suffix. Choose a~shortest nonempty shuffle-square suffix~$s$ of~$w$. Then $|s|\leqslant40$ and $s\in\mathcal M$. Write $w=us$. Since $w$ is not a~shuffle square, $s\ne w$, so~$u$ is nonempty. Factoriality gives $u\in\cC^2$. Both $w$ and~$s$ have even Parikh vector, hence so does~$u$. If $u$ were a~shuffle square, then $w=us$ would be one by~Lemma~\ref{lem:closure-shuffle}(i), a contradiction.
\end{proof}

The~reduction makes the~classification finite. Let $E_0$ be the~set of nonempty even-Parikh $C^2$-words of~length below~$40$ which are not shuffle squares, and define
\begin{align}\label{eq:C2-layers}
 E_1&=\{us:u\in E_0,\ s\in\mathcal M,\ |us|\geqslant40,\
 us\in\cC^2,\ us\notin\SSq\},\\
 E_{j+1}&=\{us:u\in E_j,\ s\in\mathcal M,\ us\in\cC^2,\
 us\notin\SSq\}\qquad(j\geqslant1).
\end{align}
Listing~\ref{lst:C2-extension} constructs precisely these sets.

\begin{lstlisting}[style=pseudocodestyle,caption={The finite extension verification for $C^2$},label={lst:C2-extension}]
E[0] <- all nonempty even-Parikh C^2 non-shuffle-squares of length < 40
B <- BAD-PREFIX-TREE(2, 40)
M <- empty set
for n <- 1,...,40 and u in B[n-1] and (*@$\word{a}$@*) in {(*@$\word{1}$@*),(*@$\word{2}$@*)} do
  v <- u*(*@$\word{a}$@*)
  if v is in C^2 and v is a shuffle square then
    insert reverse(v) into M

for j <- 0,1 do
  CANDIDATES <- empty set
  for u in E[j] and s in M do
    if u*s is in C^2 and (j > 0 or length(u*s) >= 40) then
      insert u*s into CANDIDATES
  E[j+1] <- {w in CANDIDATES : SHUFFLE-SQUARE(w) is false}
\end{lstlisting}
  
  The~exact results are summarized in~Table~\ref{tab:C2-computation}.
\begin{table}[htbp]
\centering
\caption{Summary of~the~exhaustive computation for~the~class~$C^2$}
\label{tab:C2-computation}
\begin{tabular}{@{}lr@{}}
\toprule
Quantity & Value\\
\midrule
$|E_0|$ & 208\\
Distinct $C^2$ candidates for~$E_1$ & 13\,752\\
$|E_1|$ & 4\\
Distinct $C^2$ candidates for~$E_2$ & 200\\
$|E_2|$ & 0\\
\bottomrule
\end{tabular}
\end{table}

The~four words in~$E_1$ are $X_{46},\overline{X}_{46},X_{48},\overline{X}_{48}$, where
\begin{align}
\label{eq:X46}X_{46}&=\word{1221121122122122112112112112112212212211211221}\\ \nonumber&=\word{122}(\word{112})^2(\word{212})^2\word{21}(\word{121})^4\word{12}(\word{212})^2(\word{211})^2\word{221},
  \\
\label{eq:X48}X_{48}&=\word{122112112122122121121121121121121221221211211221}\\ 
\nonumber&=\word{122}(\word{112})^2(\word{122})^2(\word{121})^6(\word{221})^2(\word{211})^2\word{221}.
\end{align}
Notice that both representatives are palindromes.

\begin{theorem}\label{thm:C2-parity}
Let $w\in\cC^2$ and $|w|>48$. Then $w$ is a~shuffle square if and only if its Parikh vector is even. The~bound~$48$ is sharp. Moreover, there are exactly~$212$ nonempty even-Parikh $C^2$-words which are not shuffle squares.
\end{theorem}

\begin{proof}
Only sufficiency requires proof. Listing~\ref{lst:parity-levels} gives all~$208$ exceptions of~length below~$40$. Lemma~\ref{lem:C2-reduction} shows that every longer exception can be reduced, by~deleting a~member of~$\mathcal M$, to~a~shorter exception. Repeating eventually reaches~$E_0$, so every exception of~length at least~$40$ belongs to~one of~the~layers in~\eqref{eq:C2-layers}. Listing~\ref{lst:C2-extension} gives $|E_1|=4$ and $E_2=\varnothing$, so there is no further layer. Thus the~complete exceptional set is $E_0\cup E_1$, of~cardinality $208+4=212$. Its longest members are $X_{48}$ and~$\overline X_{48}$, proving both the criterion above length~$48$ and sharpness.

The~finite verification consists of the~proved reduction, the~complete construction in~Listing~\ref{lst:C2-extension}, and the~five exact cardinalities displayed above. No unbounded search is being inferred from a~cutoff.
\end{proof}

For~completeness, Table~\ref{tab:C2-exceptions} gives the~number of exceptions at~every length where it is nonzero.

\begin{table}[ht]
\centering
\caption{Nonempty even-Parikh $C^2$-words which are not shuffle squares}
\label{tab:C2-exceptions}
\small
\begin{tabular}{@{}rr@{\qquad}rr@{\qquad}rr@{}}
\toprule
$n$ & $e_n$ & $n$ & $e_n$ & $n$ & $e_n$\\
\midrule
 4 & 2  & 18 & 16 & 30 & 10\\
 8 & 6  & 20 & 16 & 32 & 24\\
10 & 4  & 22 & 12 & 34 & 12\\
12 & 10 & 24 & 20 & 38 & 8\\
14 & 6  & 26 & 36 & 46 & 2\\
16 & 8  & 28 & 18 & 48 & 2\\
\midrule
\multicolumn{6}{c}{$\displaystyle\sum_n e_n=212$}\\
\bottomrule
\end{tabular}
\end{table}

\subsection{The~class \texorpdfstring{$C^1$}{C1}}

The~qualitative picture changes at~the~class studied by~Grytczuk, Pawlik, and Ruci\'nski in \cite{GrytczukPawlikRucinskiProceedings, GrytczukPawlikRucinskiPreprint}: after relabelling the~alphabet, their doubly binary words are precisely the~$C^1$-words. Their odd-ABBA words $(\word{2112})^{2j+1}$ are even-Parikh non-shuffle-squares of~unbounded length
\cite{GrytczukPawlikRucinskiPreprint}. Hence no global threshold can
make evenness sufficient throughout~$C^1$.

There is likewise no finite local prefix threshold. The~right-infinite word
\begin{align}\label{eq:C1-infinite-ray}
 Z=\word{12}(\word{1122})^\omega
\end{align}
has every finite prefix in~$\cC^1$, and every one of~its even-length prefixes has both letter multiplicities odd. Consequently, no nonempty prefix of~$Z$ is a~shuffle square. Nevertheless, even-Parikh $C^1$-words admit a~sharp finite reduction.

\begin{theorem}\label{thm:C1-reduction}
Let $w\in\cC^1$ have even Parikh vector and suppose that it is not a shuffle square. If $|w|\geqslant36$, then $w$ has a~proper nonempty shuffle-square prefix and a~proper nonempty shuffle-square suffix. The maximum length of~an~even-Parikh $C^1$-word with no such prefix is~$34$, so~$36$ is the~sharp threshold among admissible (even) lengths. Moreover, there are exactly~$230$ nonempty even-Parikh $C^1$-words having no nonempty shuffle-square prefix, and all of~them have length at~most~$34$.
\end{theorem}

\begin{proof}
Apply recurrence~\eqref{eq:Ck-bad-tree} with $k=1$. The~complete level sizes through the~onset of~periodicity are
\begin{align}\label{eq:C1-bad-counts}
(|\cB_n^{(1)}|)_{n=0}^{42}
={}&(1,2,2,4,4,6,6,10,10,14,16,26,26,36,38,56,52,72,76,104,\\ \nonumber
&108,150,146,202,186,244,242,316,298,372,346,432,390,454,\\ \nonumber
&418,470,422,430,422,422,422,422,422).
\end{align}
The~nonzero counts of~even-Parikh words in these levels are given in Table~\ref{tab:C1-even-bad}. They sum to~$230$, and the~last occurs at length~$34$.

\begin{table}[ht]
\centering
\caption{Nonempty even-Parikh $C^1$-words with no nonempty shuffle-square prefix}
\label{tab:C1-even-bad}
\small
\begin{tabular}{@{}rr@{\qquad}rr@{\qquad}rr@{}}
\toprule
$n$ & number & $n$ & number & $n$ & number\\
\midrule
 4 & 2 & 16 & 10 & 26 & 28\\
 8 & 4 & 18 & 14 & 28 & 38\\
10 & 2 & 20 & 24 & 30 & 20\\
12 & 8 & 22 & 18 & 32 & 24\\
14 & 6 & 24 & 28 & 34 & 4\\
\midrule
\multicolumn{6}{c}{$\displaystyle\text{total}=230$}\\
\bottomrule
\end{tabular}
\end{table}

It remains to~prove that no later even-Parikh level can reappear. For a~residual~$q$ and a~word~$z$ whose first run starts after a~run boundary, write
\begin{align*}
 q\xrightarrow{z}r
\end{align*}
when the~unpruned run-capacity transitions of Listing~\ref{lst:run-test}, initialized at~$q$, can finish at~$r$. The~finite calculation partitions the~$422$ words in $\cB_{38}^{(1)}$ into two classes of~size~$211$. For~each seed~$u$ in the~first class, $u$ ends in~$\word{22}$, both its multiplicities are odd, the~residual~$\word{12}$ is reachable after~$u$, and
\begin{align}\label{eq:C1-pump-one}
 \word{12}\xrightarrow{\word{1122}}\word{12},
 \qquad \word{12}\xrightarrow{\word{12}}\eps,
 \qquad \word{12}\xrightarrow{\word{1121}}\eps.
\end{align}
For~each seed in~the~second class, the~complementary assertions hold: $u$ ends in~$\word{11}$, residual~$\word{21}$ is reachable, and
\begin{align}\label{eq:C1-pump-two}
 \word{21}\xrightarrow{\word{2211}}\word{21},
 \qquad \word{21}\xrightarrow{\word{21}}\eps,
 \qquad \word{21}\xrightarrow{\word{2212}}\eps.
\end{align}

For~a~concrete illustration, consider the~seed
\begin{align*}
 U=\word{12}(\word{1122})^9\in\cB_{38}^{(1)}.
\end{align*}
It contains $19$ copies of~each letter and reaches residual~$\word{12}$. Appending two periods and then the~exit word~$\word{12}$ gives the~sequence of~transitions
\begin{align*}
 \eps\xrightarrow{U}\word{12}
 \xrightarrow{\word{1122}}\word{12}
 \xrightarrow{\word{1122}}\word{12}
 \xrightarrow{\word{12}}\eps.
\end{align*}
Consequently, the~word
\begin{align*}
 W=U(\word{1122})^2\word{12}
   =\word{12}(\word{1122})^{11}\word{12}
\end{align*}
is a~shuffle square of~length~$48$, with root~$(\word{12})^{12}$. An~explicit decomposition assigns the~initial~$\word{12}$ to~the~first copy, the~first and third positions of~each block~$\word{1122}$ to~the~first copy, the~second and fourth positions to~the~second copy, and the~final~$\word{12}$ to~the~second copy. After~$U$ and after each appended period, the~first partial copy exceeds the~second by~the~same suffix~$\word{12}$; the~exit word supplies precisely this missing suffix.

Returning to~the~general argument, consider any seed~$u$ in~the~first class. The~$C^1$ condition forces the~next letter after~$\word{22}$ to~be~$\word{1}$. Choosing~$\word{2}$ next completes the~exit word~$\word{12}$ and hence a~shuffle-square prefix by~\eqref{eq:C1-pump-one}. Choosing $\word{1}$ is therefore the~only bad continuation. The~next letter is then forced to~be~$\word{2}$. Choosing~$\word{1}$ after it completes the~other exit $\word{1121}$ and gives a~shuffle-square prefix. Hence the~unique bad four-letter continuation is~$\word{1122}$. This block returns both the~last run and the~reachable residual to~their initial states. Induction makes $u(\word{1122})^\omega$ the~unique bad ray from~$u$. Every finite prefix along this ray has at least one odd multiplicity, while each block endpoint has both multiplicities odd, so no full word on the~ray is itself a~shuffle square. The~same argument with~\eqref{eq:C1-pump-two} gives the unique ray $u(\word{2211})^\omega$ from every seed in~the~second class. Thus $|\cB_n^{(1)}|=422$ for~every $n\geqslant38$, and none of~these words has even Parikh vector.

It follows from Table~\ref{tab:C1-even-bad} that an~even-Parikh $C^1$-word of~length at least~$36$ with no shuffle-square prefix does not exist. A~non-shuffle-square word satisfying the~hypotheses must therefore have a~proper such prefix. Applying the~same conclusion to $\rev(w)$ and using reversal closure gives the~suffix. Finally, \begin{align}\label{eq:C1-sharp34}
 \word{1221121121221221212211221122112212}
 =\word{122}(\word{112})^2(\word{122})^2\word{12}(\word{1221})^4\word{2}
\end{align}
has length~$34$, even Parikh vector, and belongs to $\cB_{34}^{(1)}$, proving the~sharpness assertion.

Listing~\ref{lst:C1-verification} records all finite assertions needed for the~count, the~pumping step, and the~example establishing sharpness.

\begin{lstlisting}[style=pseudocodestyle,caption={The finite count and pumping argument for $C^1$},label={lst:C1-verification}]
B <- BAD-PREFIX-TREE(1, 42)
assert (size(B[n]))[n=0,...,42] equals the vector in (C1-bad-counts)
assert the even-Parikh counts equal Table C1-even-bad and sum to 230

for u in B[38] do
  V <- {v in B[42] : u is a prefix of v}
  assert size(V) = 1
  p <- the last four letters of the unique word in V
  if p = (*@$\word{1122}$@*) then
    assert u ends in (*@$\word{22}$@*)
    (q, exits) <- ((*@$\word{12}$@*), {(*@$\word{12}$@*),(*@$\word{1121}$@*)})
  else
    assert p = (*@$\word{2211}$@*) and u ends in (*@$\word{11}$@*)
    (q, exits) <- ((*@$\word{21}$@*), {(*@$\word{21}$@*),(*@$\word{2212}$@*)})
  assert both multiplicities of u are odd
  assert q is reachable after u
  assert q --p--> q and q --z--> epsilon for every z in exits
assert 211 seeds have p = (*@$\word{1122}$@*) and 211 have p = (*@$\word{2211}$@*)
assert the word in (C1-sharp34) belongs to B[34]
\end{lstlisting}

The~recurrence, complete level vector, Table~\ref{tab:C1-even-bad}, transition identities \eqref{eq:C1-pump-one}--\eqref{eq:C1-pump-two}, and Listing~\ref{lst:C1-verification} together give a~finite verification independent of supplementary source code.
\end{proof}

The~theorem yields a~finite terminal core, though unlike at levels $C^2$ and~$C^3$ the~individual shuffle-square blocks cannot be bounded.

\begin{corollary}\label{cor:C1-factorization}
Every nonempty even-Parikh $C^1$-word~$w$ has a~greedy factorization
\begin{align*}
 w=S_1S_2\cdots S_k r,
\end{align*}
where every $S_i$ is a~$C^1$ shuffle square and the~terminal remainder is either empty or one of~the~$230$ words counted in Table~\ref{tab:C1-even-bad}. In~particular, $|r|\leqslant34$. If $w$ is not a~shuffle square, then $r\ne\eps$. There is no uniform upper bound on the~lengths of~the~blocks~$S_i$.
\end{corollary}

\begin{proof}
Repeatedly remove the~shortest nonempty shuffle-square prefix. Every remainder is a~$C^1$ factor and retains an~even Parikh vector. At termination, a~nonempty remainder has no nonempty shuffle-square prefix and is therefore one of the~$230$ words in Table~\ref{tab:C1-even-bad}. Concatenation of~the~$S_i$ proves that $r=\eps$ would make~$w$ a~shuffle square.

For~the~absence of~a block bound, put
\begin{align*}
 U=\word{12}(\word{1122})^9,
 \qquad Y_m=U(\word{1122})^m\word{12}\quad(m\geqslant0).
\end{align*}
The~word~$U$ is one of~the first-class seeds above and reaches residual~$\word{12}$. The~pumping and exit identities in \eqref{eq:C1-pump-one} show that every~$Y_m$ is a~$C^1$ shuffle square, whereas every proper nonempty prefix lies on the~bad ray from~$U$. Thus $Y_m$ is a~minimal shuffle-square prefix of~length~$40+4m$.
\end{proof}

\begin{example}
\label{ex:odd-ABBA-reduction}
Consider the~odd members of~the~family $W_n=(\word{2112})^n$ introduced above. Grytczuk, Pawlik, and Ruci\'nski proved that none of~these words is a~shuffle square~\cite{GrytczukPawlikRucinskiPreprint}. Put
\begin{align*}
  S=(\word{2112})^2=\word{21122112}.
\end{align*}
For~$j\geqslant1$, we have $W_{2j+1}=S W_{2j-1}$. The~word~$S$ is an~ordinary square and hence a~shuffle square. It is the~shortest nonempty shuffle-square prefix: the~prefixes of~lengths~$2$ and~$6$ have odd letter multiplicities, while the~length-$4$ prefix is~$\word{2112}$, the~complement of~the~obstruction in~Example~\ref{ex:1221}. The~greedy factorization therefore removes~$S$ exactly~$j$ times and leaves~$W_1=\word{2112}$. Thus this infinite family of~$C^1$ exceptions reduces to~a~single irreducible word of~length~$4$.
\end{example}

These $C^1$ statements do not characterize all doubly binary even-Parikh non-shuffle-squares, but they separate a~finite irreducible core from the~remaining words and locate exactly where the~bounded local method valid for~$C^2$ and above breaks down.

\section{Bounded local factorizations}

\noindent
The~finite length-$24$ lemma also has consequences which do not assume an~even Parikh vector. We state them at the~$C^3$ level, so all stronger differentiability classes follow at once.

\begin{proposition}
\label{prop:boundary-local}
Let $w=\word{w}_1\cdots \word{w}_n$ be a~finite $C^3$-word.
\begin{enumerate}[(i)]
\item For~every $1\leqslant i\leqslant n-23$, some factor beginning at~position~$i$ is a~shuffle square of~length in~$\{2,4,\ldots,24\}$.
\item For~every $24\leqslant j\leqslant n$, some factor ending at~position~$j$ is a~shuffle square of~length in~$\{2,4,\ldots,24\}$.
\end{enumerate}
Consequently, a~$C^3$-word of~length~$n\geqslant24$ contains at~least $n-23$ shuffle-square occurrences with~distinct starting positions.
\end{proposition}

\begin{proof}
For~(i), the~length-$24$ factor $\word{w}_i\word{w}_{i+1}\cdots \word{w}_{i+23}$ belongs to~$\cC^3$ by~Lemma~\ref{lem:smooth-closure}(i). Apply the~prefix part of~Lemma~\ref{lem:C3-local}. Part~(ii) follows from~the~suffix part applied to~$\word{w}_{j-23}\cdots \word{w}_j$. The~final count chooses one occurrence for~each of~the~$n-23$ possible starting positions in~(i). Occurrences with~different starting positions are different intervals.
\end{proof}

\begin{theorem}
\label{thm:finite-factorization}
Every finite $C^3$-word~$w$ has a~factorization
\begin{align*}
  w=S_1S_2\cdots S_k r,
\end{align*}
obtained by~repeatedly removing the~shortest nonempty shuffle-square prefix. Here
\begin{enumerate}[(i)]
\item each $S_i$ is a~$C^3$ shuffle square and $2\leqslant|S_i|\leqslant24$.
\item the~terminal remainder~$r$ belongs to~$\cC^3$, has no nonempty shuffle-square prefix, satisfies $r\in\cB_{|r|}^{(3)}$, and has $|r|\leqslant23$.
\item $S_1\cdots S_k$ is a~shuffle square whenever $k\geqslant1$.
\end{enumerate}
If $w\in\cC^h$ for~some $h\geqslant3$, then all blocks and the~remainder belong to~$\cC^h$. If $w$ is smooth, they are smooth. In~particular, every $C^3$-word of~length~$n\geqslant24$ has a~shuffle-square prefix and suffix of~length at~least $n-23$.
\end{theorem}

\begin{proof}
Start with~the~remainder $r_0=w$. If $r_{i-1}$ has a~nonempty shuffle-square prefix, let $S_i$ be its shortest such prefix and write $r_{i-1}=S_ir_i$. The~choices are deterministic, and the~length strictly decreases, so the~process terminates in~a~remainder~$r$ with~no such prefix. Every intermediate remainder and every selected block is a~factor of~$w$, hence belongs to~$\cC^3$ by~Lemma~\ref{lem:smooth-closure}(i). The~same lemma preserves any stronger class~$\cC^k$ and smoothness.

If an~intermediate remainder has length at~least~$24$, its first~$24$ letters form a~$C^3$-word and Lemma~\ref{lem:C3-local} supplies a~shuffle-square prefix of~length at~most~$24$. Thus every chosen shortest block has length at~most~$24$, and a~terminal remainder cannot have length at~least~$24$. The~lower bound~$2$ follows from~Lemma~\ref{lem:even-necessary}. Repeated use of~Lemma~\ref{lem:closure-shuffle}(i) proves~(iii).

When $n\geqslant24$, at~least one block is removed and $|S_1\cdots S_k|=n-|r|\geqslant n-23$. Applying the~established prefix statement to~$\rev(w)$, then using Lemmas~\ref{lem:smooth-closure}(ii) and~\ref{lem:closure-shuffle}(ii), gives the~suffix statement.
\end{proof}

\begin{remark}
The~constant~$23$ in~Theorem~\ref{thm:finite-factorization} is optimal even for~smooth words. Each word in~\eqref{eq:bad23} is smooth and has no nonempty shuffle-square prefix, so its greedy factorization has no block and leaves the~entire length-$23$ word as~its remainder.
\end{remark}

\begin{example}
\label{ex:greedy-P20}
For~the~smooth word $P_{20}$, Theorem~\ref{thm:main-intro} shows that the~shortest shuffle-square prefix is $P_{16}$, and the~greedy factorization is
\begin{align*}
  P_{20}=P_{16}\,\word{1221}.
\end{align*}
The~remainder $\word{1221}$ has no nonempty shuffle-square prefix, by~Example~\ref{ex:1221}.  Nevertheless, $P_{20}$ itself is a~shuffle square, as~Table~\ref{tab:base} shows.  Thus Theorem~\ref{thm:finite-factorization} is a~canonical bounded-block construction, not a~claim that its terminal remainder characterizes whether the~original word is a~shuffle square.
\end{example}

The~preceding construction needs only a~local hypothesis in~the~infinite setting.

\begin{definition}
\label{def:24-local}
A~right-infinite or bi-infinite binary word~$X$ is \emph{length-$24$ locally~$C^3$} if every factor of~$X$ of~length~$24$ belongs to~$\cC^3$. Any word whose length-$24$ factors belong to~$\cC^k$ for~some $k\geqslant3$ has this property. Proposition~\ref{prop:infinite-factors-smooth} shows the~same for~every smooth right-infinite word.
\end{definition}

\begin{theorem}
\label{thm:infinite-factorization}
Let $X=X_1X_2\cdots$ be a~length-$24$ locally~$C^3$ right-infinite word, and fix a~starting position $t\geqslant1$.  Put $s_0=t-1$.  Recursively, let $d_k$ be the~least positive integer~$d$ for~which
\begin{align*}
  X_{s_{k-1}+1}X_{s_{k-1}+2}\cdots X_{s_{k-1}+d}
\end{align*}
is a~shuffle square, and put $s_k=s_{k-1}+d_k$.  Then the~recursion is defined for~all $k\geqslant1$ and
\begin{align*}
  2\leqslant d_k\leqslant24.
\end{align*}
Every cumulative factor $X_tX_{t+1}\cdots X_{s_k}$ is a~shuffle square. Moreover, if
\begin{align*}
 A_{X,t}(N)=\#\{1\leqslant\ell\leqslant N:
 X_tX_{t+1}\cdots X_{t+\ell-1}\text{ is a~shuffle square}\},
\end{align*}
then
\begin{align}\label{eq:local-density}
  \left\lfloor\frac{N}{24}\right\rfloor
  \leqslant A_{X,t}(N)\leqslant\left\lfloor\frac{N}{2}\right\rfloor
  \qquad(N\geqslant1),
\end{align}
and hence
\begin{align*}
  \frac1{24}
  \leqslant\liminf_{N\to\infty}\frac{A_{X,t}(N)}{N}
  \leqslant\limsup_{N\to\infty}\frac{A_{X,t}(N)}{N}
  \leqslant\frac12.
\end{align*}
\end{theorem}

\begin{proof}
Once $s_{k-1}$ is defined, the~length-$24$ factor beginning at~$s_{k-1}+1$ belongs to~$\cC^3$. Lemma~\ref{lem:C3-local} gives a~nonempty shuffle-square prefix, so the~least~$d_k$ exists and is at~most~$24$. It is at~least~$2$ by~parity. The~cumulative factor is the~concatenation of~the~first $k$ shuffle-square blocks, hence is a~shuffle square by~Lemma~\ref{lem:closure-shuffle}(i).

Writing $L_k=s_k-t+1=d_1+\cdots+d_k$, we have $L_k\leqslant24k$. For~$k=\lfloor N/24\rfloor$, the~$k$ distinct lengths $L_1,\ldots,L_k$ are all at~most~$N$ (with~the~assertion empty when $k=0$). This proves the~lower bound in~\eqref{eq:local-density}.  Its upper bound follows from~Lemma~\ref{lem:even-necessary}, since only even lengths can be shuffle-square lengths.  Dividing by~$N$ and taking lower and upper limits gives the~final chain of~inequalities.
\end{proof}

\begin{corollary}
\label{cor:bi-infinite}
Let $X=(X_i)_{i\in\mathbb Z}$ be a~length-$24$ locally~$C^3$ bi-infinite word, and fix any cut $c_0\in\mathbb Z$.  There is a~strictly increasing sequence of~cuts $(c_k)_{k\in\mathbb Z}$ such that
\begin{align*}
  2\leqslant c_{k+1}-c_k\leqslant24
\end{align*}
and every block $X_{c_k+1}\cdots X_{c_{k+1}}$ is a~shuffle square.
\end{corollary}

\begin{proof}
Starting at~$c_0$, construct the~cuts to~the~right by~the~proof of~Theorem~\ref{thm:infinite-factorization}. Given a~cut on~the~leftward side, apply the~suffix part of~Lemma~\ref{lem:C3-local} to~the~length-$24$ factor ending at~that cut and move left by~the~length of~a~shortest shuffle-square suffix. This constructs all negative-indexed cuts, with~the~same bounds.
\end{proof}

\begin{corollary}\label{cor:C2-infinite-factorization}
Call a~right-infinite or bi-infinite word \emph{length-$40$ locally~$C^2$} when each of~its length-$40$ factors belongs to~$\cC^2$. Then all conclusions of~Theorem~\ref{thm:infinite-factorization} and Corollary~\ref{cor:bi-infinite} remain valid after replacing ``length-$24$ locally~$C^3$'' by~``length-$40$ locally~$C^2$'', the~upper block bound~$24$ by~$40$, and the~lower density bound~$1/24$ by~$1/40$.
\end{corollary}

\begin{proof}
The~proofs are unchanged, with Lemma~\ref{lem:C2-local} supplying a~shuffle-square prefix or suffix of~length at~most~$40$ at~each step.
\end{proof}

\section{Characterization of~all shuffle-square prefixes}\label{sec:kolakoski-prefixes}

\noindent
Example~\ref{ex:first-decomposition} gives an~explicit shuffle decomposition of~the~first shuffle-square prefix, of~length~$16$.  Table~\ref{tab:base} gives such decompositions for~the~next three. In~each row, $I$ is a~support of~one copy and its complement in~$\{1,\ldots,m\}$ supports the~other copy.  The~two induced subwords both equal the~displayed root.  These three decompositions are illustrative direct checks, since the~universal theorem already covers their lengths.

\begin{table}[ht]
\centering
\caption{Further explicit shuffle decompositions of~Kolakoski prefixes}
\label{tab:base}
\small
\begin{tabular}{@{}ccl@{}}
\toprule
$m$ & root & support $I$\\
\midrule
$20$ & $\word{1221121221}$
 & $\{1,2,3,4,5,6,7,8,9,16\}$\\
$24$ & $\word{122112122112}$
 & $\{1,2,3,4,5,6,7,8,9,16,20,21\}$\\
$28$ & $\word{12211212211221}$
 & $\{1,2,3,4,5,6,7,8,9,16,20,21,24,25\}$\\
\bottomrule
\end{tabular}
\end{table}

Direct counting in~\eqref{eq:K} gives
\begin{align}\label{eq:small-EPK}
  \EPK\cap\{1,\ldots,33\}=\{4,8,16,20,24,28\}.
\end{align}
Examples~\ref{ex:1221} and~\ref{ex:length8} show that the~first two words in~\eqref{eq:small-EPK} are not shuffle squares.  Example~\ref{ex:first-decomposition} and Table~\ref{tab:base} handle the~other four.

We can now prove the~exact characterization announced in~Theorem~\ref{thm:main-intro}. This characterization was also obtained independently by~Kilpatrick~\cite[Theorem~1.3]{Kilpatrick}.

\begin{proof}[Proof of~Theorem~\ref{thm:main-intro}]
If $P_m$ is a~shuffle square, then $m\in\EPK$ by~Lemma~\ref{lem:even-necessary}.  The~two exceptional values are established in~Examples~\ref{ex:1221} and~\ref{ex:length8}.

Conversely, let $m\in\EPK\setminus\{4,8\}$. Equation~\eqref{eq:small-EPK} shows that either $m=16$ or $m\geqslant20$. The~case $m=16$ is Example~\ref{ex:first-decomposition}. In~the~remaining cases, $P_m$ is a~smooth factor of~$\K$ by~Corollary~\ref{cor:K-factors-smooth}, hence belongs to~$\cC^3$. It has length greater than~$16$ and even Parikh vector. It is therefore a~shuffle square by~Theorem~\ref{thm:C3-parity}.
\end{proof}

\begin{corollary}\label{cor:kolakoski-density}
The~increasing enumeration of~$\EPK\setminus\{4,8\}$, equivalently the~sequence of~shuffle-square prefix lengths of~$\K$, begins
\begin{align*}
  16,20,24,28,34,36,40,44,48,56,60,64,66,68,72,76,80,\ldots.
\end{align*}
It is infinite, and consecutive terms differ by~one of~$2,4,6,8$. If
\begin{align*}
  H_{\K}(N)=\#\{1\leqslant m\leqslant N:P_m\text{ is a~shuffle square}\},
\end{align*}
then, for~every $N\geqslant16$,
\begin{align}\label{eq:K-density}
  H_{\K}(N)\geqslant
  1+\left\lfloor\frac{N-16}{8}\right\rfloor,
\end{align}
and consequently
\begin{align*}
  \liminf_{N\to\infty}\frac{H_{\K}(N)}N\geqslant\frac18.
\end{align*}
\end{corollary}

\begin{proof}
The~displayed sequence follows from~Theorem~\ref{thm:main-intro}, and~\eqref{eq:EPKterms}. Infinitude and the~gap bound follow from~Proposition~\ref{prop:even-gap}. If the~displayed lengths are $\word{a}_1<\word{a}_2<\cdots$, then $\word{a}_1=16$ and $\word{a}_j\leqslant16+8(j-1)$.  Taking $j=1+\lfloor(N-16)/8\rfloor$ proves~\eqref{eq:K-density}.  The~lower-density bound follows after division by~$N$.
\end{proof}

\begin{corollary}
\label{cor:constructed-subsequence}
Put $s_0=0$.  Once $s_{k-1}$ is defined, let $d_k$ be the~least positive integer~$d$ for~which
\begin{align*}
  \K_{s_{k-1}+1}\K_{s_{k-1}+2}\cdots\K_{s_{k-1}+d}
\end{align*}
is a~shuffle square, and put $s_k=s_{k-1}+d_k$.  This recursion is well-defined for~every $k\geqslant1$, and
\begin{align*}
  2\leqslant d_k\leqslant24.
\end{align*}
Every $P_{s_k}$ is a~shuffle square.  The~initial cumulative endpoints are
\begin{align*}
  16,34,36,48,56,66,68,88,90,104,110,116,124,134,136,\ldots.
\end{align*}
They form an~infinite proper subsequence of~the~increasing enumeration of~$\EPK\setminus\{4,8\}$.  The~omitted terms, such as~$20,24,28,40$, are shuffle-square prefixes but are not endpoints of~that particular recursive factorization.
\end{corollary}

\begin{proof}
Corollary~\ref{cor:K-factors-smooth} makes~$\K$ length-$24$ locally~$C^3$. Theorem~\ref{thm:infinite-factorization}, with~$X=\K$ and $t=1$, proves that the~recursion is defined, gives the~block-length bounds, and makes every cumulative prefix a~shuffle square. Hence $s_k\in\EPK\setminus\{4,8\}$ for~every $k\geqslant1$ by~Theorem~\ref{thm:main-intro}.

The~initial values above are obtained by~successively applying the~exact recognizer in~Listing~\ref{lst:residual-test} to~the~even-length candidate blocks.

Finally, $s_1=16$ and $P_{20}=P_{16}\word{1221}$. The~word $\word{1221}$ has no nonempty shuffle-square prefix by~Example~\ref{ex:1221}, so~$20$ is not a~recursive endpoint, although $P_{20}$ is a~shuffle square by~Theorem~\ref{thm:main-intro}. Hence the~subsequence is proper.
\end{proof}

\begin{example}
The~first block is $P_{16}$ from~Example~\ref{ex:first-decomposition}.  The~next block, written with~positions relative to~that block, is
\begin{align*}
  \K_{17}\cdots\K_{34}=\word{122121121221121121}.
\end{align*}
The~support $I=\{1,2,3,4,5,6,7,8,12\}$ and its complement both induce the~root $\word{122121121}$. The~third block is simply $\K_{35}\K_{36}=\word{22}$, with~root~$\word{2}$. Hence the~first three cumulative endpoints are $16,34,36$.

Applying Listing~\ref{lst:residual-test} successively to~each shorter even-length candidate verifies that none works at~any displayed step.

Together with~Lemma~\ref{lem:even-necessary}, this excludes every shorter block.
\end{example}

\section{Morphisms reflecting shuffle squares}\label{sec:morphisms}

\noindent
In~this section the~alphabets need not be binary. Let $A,B$ be finite alphabets, and extend the~notation $\SSq$ to~the~nonempty shuffle squares over whichever alphabet is under consideration. A~morphism $h:A^*\to B^*$ is \emph{nonerasing} if $h(\word{a})\ne\eps$ for~every $\word{a}\in A$. Every such morphism preserves shuffle squares: applying $h$ to~the~two copies of~a~root gives a~decomposition of~the~image. We say that $h$ \emph{reflects shuffle squares} if the~converse also holds, that is,
\begin{align}\label{eq:morph:reflection}
 h(w)\in\SSq\quad\Longleftrightarrow\quad w\in\SSq
 \qquad(w\in A^*).
\end{align}
The~empty word causes no exception to~this convention, since neither side holds for~$w=\eps$.

We first obtain a~necessary condition and a~complete answer when the~alphabets have the~same cardinality. For~larger target alphabets, we then derive criteria using the~ordered-graph method of~Grytczuk, Pawlik, and Ruci\'nski~\cite{GrytczukPawlikRucinskiPreprint}. The~graph representation and the~path--cycle lemma used below are taken from~that work. Regular bipartite matchings and odd-power reduction will allow us to~extract a~shuffle decomposition of~the~source word.

\subsection{Odd powers and private symbols}

\begin{lemma}\label{lem:morph:odd-powers}
Let $k_{\word{a}}$ be a~positive odd integer for~each $\word{a}\in A$. The~morphism $D:A^*\to A^*$ given by~$D(\word{a})=\word{a}^{k_{\word{a}}}$ reflects shuffle squares.
\end{lemma}

\begin{proof}
Only reflection requires proof. Suppose that $D(w)$ is partitioned into two copies of~a~word~$R$. Write $k_{\word{a}}=2s_{\word{a}}+1$. In~each copy of~$R$, number the~occurrences of~$\word{a}$ separately and retain exactly those whose numbers are congruent to~$s_{\word{a}}+1$ modulo~$k_{\word{a}}$. Make this selection for~every letter. The~two retained subwords are identical.

We show that exactly one position survives in~each original block~$\word{a}^{k_{\word{a}}}$. At~a~boundary between original blocks, suppose that $j$ blocks of~letter~$\word{a}$ have been processed. If the~two copies have used $u$ and~$v$ occurrences of~$\word{a}$, then $u+v=k_{\word{a}} j$. The~number of~retained occurrences in~a~prefix containing $t$ occurrences is
\begin{align*}
 F_{\word{a}}(t)=\left\lfloor\frac{t+s_{\word{a}}}{k_{\word{a}}}\right\rfloor.
\end{align*}
For~all nonnegative $u,v$ with $u+v=k_{\word{a}} j$, one has
\begin{align}\label{eq:morph:rounding}
 F_{\word{a}}(u)+F_{\word{a}}(v)=j.
\end{align}
Indeed, write $u=k_{\word{a}} t+r$ with $0\le r<k_{\word{a}}$. If $r\le s_{\word{a}}$, the~two remaining floor terms are both zero. If $r>s_{\word{a}}$, they are $1$ and~$-1$. Thus the~total number of~retained positions increases by~exactly one while an~original block of~letter~$\word{a}$ is read. The~retained word is therefore~$w$, with a~partition into two identical subwords.
\end{proof}

\begin{definition}\label{def:morph:private}
A~symbol $\word{b}\in B$ is \emph{private to~$\word{a}$} for~a~morphism~$h$ if $\word{b}$ occurs in~$h(\word{a})$ and does not occur in~$h(\word{c})$ for~any $\word{c}\in A\setminus\{\word{a}\}$.
\end{definition}

\begin{proposition}\label{prop:morph:private-necessary}
If a~nonerasing morphism $h:A^*\to B^*$ reflects shuffle squares, then every $\word{a}\in A$ has a~private symbol. In~particular, $|B|\ge|A|$.
\end{proposition}

\begin{proof}
Suppose that $\word{a}$ has no private symbol, and put $V=h(\word{a})$. Choose a~nonempty word $z\in(A\setminus\{\word{a}\})^*$ such that every symbol occurring in~$V$ occurs in~$h(z)$. For~some positive~$k$, the~word~$V$ is a~subword of~$U=h(z)^k$: for~example, one can use at~most one required symbol from~each successive copy of~$h(z)$. Let~$T$ be the~complementary subword, so that $U$ is a~shuffle of~$V$ and~$T$.

The~word $VUUV$ is a~shuffle square with root~$VTV$. One copy uses the~initial~$V$, the~$T$-positions of~the~first~$U$, and the~$V$-positions of~the~second~$U$. The~other uses the~$V$-positions of~the~first~$U$, the~$T$-positions of~the~second~$U$, and the~terminal~$V$. However,
\begin{align*}
 VUUV=h(\word{a}z^{2k}\word{a}),
\end{align*}
and $\word{a}z^{2k}\word{a}$ is not a~shuffle square. Its only two occurrences of~$\word{a}$ are at~its ends. A~root would contain exactly one~$\word{a}$, which would have to~be both its first and its last letter, although the~root has length greater than~one. This is a~contradiction. Private symbols of~different letters are distinct, giving the~cardinality assertion.
\end{proof}

\begin{proposition}\label{prop:morph:private-odd}
Suppose that for~each $\word{a}\in A$ there is a~private symbol $\word{b}_{\word{a}}$ whose multiplicity $k_{\word{a}}=|h(\word{a})|_{\word{b}_{\word{a}}}$ is odd. Then $h$ reflects shuffle squares.
\end{proposition}

\begin{proof}
Define a~morphism $\rho:B^*\to A^*$ by~$\rho(\word{b}_{\word{a}})=\word{a}$ and by~erasing the~remaining symbols. Then $(\rho\circ h)(\word{a})=\word{a}^{k_{\word{a}}}$. If $h(w)$ is a~shuffle square for~nonempty~$w$, its projection $\rho(h(w))$ is a~nonempty shuffle square. Lemma~\ref{lem:morph:odd-powers} gives $w\in\SSq$.

We record a~stronger feature of~this proof for~later use. Fix any decomposition of~$h(w)$ into two equal copies and pair corresponding positions. Within each copy retain the~occurrences of~$\word{b}_{\word{a}}$ of~ranks congruent to~$(k_{\word{a}}+1)/2$ modulo~$k_{\word{a}}$. Formula~\eqref{eq:morph:rounding}, applied at~the~boundaries of~the~images $h(\word{c})$, shows that this retains exactly one position from~every image of~$\word{a}$ and none from~other images. The~retained positions are unions of~whole pairs from~the~fixed decomposition. This selection can be performed for~any chosen subset of~the~source letters.
\end{proof}

\begin{corollary}\label{cor:morph:equal-alphabets}
Let $|A|=|B|$. A~nonerasing morphism $h:A^*\to B^*$ reflects shuffle squares if and only if
\begin{align*}
 h(\word{a})=\pi(\word{a})^{k_{\word{a}}}\qquad(\word{a}\in A),
\end{align*}
where $\pi:A\to B$ is a~bijection and every $k_{\word{a}}$ is positive and odd.
\end{corollary}

\begin{proof}
By~Proposition~\ref{prop:morph:private-necessary}, choose a~private symbol $\pi(\word{a})$ for~each~$\word{a}$. These exhaust~$B$, so no symbol other than $\pi(\word{a})$ can occur in~$h(\word{a})$. Its multiplicity must be odd, since otherwise the~one-letter word~$\word{a}$ would have a~shuffle-square image. The~converse follows from~Lemma~\ref{lem:morph:odd-powers} and invariance under relabelling.
\end{proof}

\begin{example}\label{ex:morph:private}
The~morphism given by~$h(\word{a})=\word{2137}$ and $h(\word{b})=\word{4137}$ reflects shuffle squares: the~private symbols $\word{2}$ and~$\word{4}$ each have multiplicity~$1$, while the~symbols $\word{1},\word{3},\word{7}$ occur in~both images. In~contrast, the~mere existence of~private symbols is not sufficient. The~morphism $\word{a}\mapsto\word{11}$, $\word{b}\mapsto\word{2}$ has private symbols for~both letters but maps~$\word{a}$ to~a~shuffle square. The~next criteria include cases in~which every symbol in~every image has even multiplicity.
\end{example}

\subsection{Extraction from ordered graphs}

An~\emph{ordered graph} is a~graph whose vertex set is equipped with a~linear order. Loops and parallel edges are allowed. For~an~edge~$e$, write $\min e$ and~$\max e$ for~its smaller and larger endpoints, respectively; these coincide when~$e$ is a~loop. Two edges form a~\emph{nest} if, after possibly interchanging their names,
\begin{align*}
 \min e<\min f\le\max f<\max e.
\end{align*}
Thus both endpoints of~$f$ lie strictly between the~endpoints of~$e$. If the~vertices are placed in~order on~a~horizontal line and edges are drawn as~semicircular arcs above it, the~arc corresponding to~$f$ lies inside the~arc corresponding to~$e$. The~weak middle inequality allows the~inner edge~$f$ to~be a~loop.

An~ordered graph is \emph{nest-free} if it contains no such pair of~edges in~the~terminology of~\cite{GrytczukPawlikRucinskiPreprint}. For~example, on~vertices $1<2<3<4$, the~edges $\{1,4\}$ and~$\{2,3\}$ form a~nest, whereas $\{1,3\}$ and~$\{2,4\}$ do not: their endpoints interleave. In~particular, nest-free graphs may contain crossing edges.

A~word is a~shuffle square precisely when its positions admit a~nest-free perfect matching that pairs equal symbols. For~completeness, pair equal-rank positions of~two copies of~a~root. Both the~smaller and the~larger endpoint of~these pairs increase with~the~rank, so no nest occurs. Conversely, order the~edges of~such a~matching by~their smaller endpoints. Their larger endpoints also increase, and the~two endpoint sequences induce identical words. Contracting consecutive equal-letter positions into runs preserves the~absence of~nests. The~degree of~each resulting vertex, counting a~loop twice, equals the~length of~its run. This is the~run-graph representation of~\cite[Proposition~1]{GrytczukPawlikRucinskiPreprint}.

We use the~following result.

\begin{lemma}[Grytczuk--Pawlik--Ruci\'nski~\cite{GrytczukPawlikRucinskiPreprint}, Lemma~1]\label{lem:morph:path-cycle}
Let an~ordered nest-free graph contain a~path~$P$ joining its first and last vertices. Every cycle vertex-disjoint from~$P$ has even length.
\end{lemma}

\begin{lemma}\label{lem:morph:run-extraction}
Let $H$ be a~word on~two distinct symbols~$\word{x},\word{y}$, containing~$\word{y}$ and beginning and ending with~$\word{x}$. Suppose that its first and last $\word{x}$-runs have odd lengths, every other $\word{x}$-run has even length, and all $\word{y}$-runs have the~same positive length~$q$. Given any nest-free perfect matching of~equal symbols on~the~positions of~$H$, one can select a~subset of~its edges that uses exactly one position from~each $\word{y}$-run and no $\word{x}$-positions.
\end{lemma}

\begin{proof}
Contract the~runs. The~resulting nest-free multigraph is the~disjoint union of~a~graph $G_{\word{x}}$ on~the~$\word{x}$-runs and a~graph $G_{\word{y}}$ on~the~$\word{y}$-runs, with~the~vertex order inherited from~$H$. Exactly the~first and last vertices of~$G_{\word{x}}$ have odd degree. They belong to~the~same connected component, since every component has an~even number of~odd-degree vertices. Hence $G_{\word{x}}$ contains a~path~$P$ joining the~extreme vertices of~the~whole graph.

The~graph $G_{\word{y}}$ has no loop: a~path from~the~first to~the~last vertex must have an~edge spanning the~vertex of~any such loop, which would create a~nest. By~Lemma~\ref{lem:morph:path-cycle}, $G_{\word{y}}$ also has no odd cycle of~length at~least three. Thus $G_{\word{y}}$ is bipartite. This is the~parity mechanism used in~the~proof of~\cite[Theorem~1]{GrytczukPawlikRucinskiPreprint}.

Every vertex of~$G_{\word{y}}$ has degree~$q$. If its bipartition is $L\mathbin{\dot\cup}R$, then counting edges gives $|L|=|R|$ and, for~each $S\subseteq L$,
\begin{align*}
 q|S|\le q|N(S)|.
\end{align*}
Hall's theorem therefore gives a~perfect matching of~$G_{\word{y}}$. For~each edge of~this matching choose its corresponding edge in~the~original position matching. Exactly one position in~each $\word{y}$-run is selected, as~required.
\end{proof}

\begin{theorem}\label{thm:morph:multirun}
For~each $\word{a}\in A$, choose distinct symbols $\word{x}_{\word{a}},\word{y}_{\word{a}}$, with~the~pairs $\{\word{x}_{\word{a}},\word{y}_{\word{a}}\}$ disjoint for~different letters. Let
\begin{align}\label{eq:morph:multirun}
 h(\word{a})=\word{x}_{\word{a}}^{p_{\word{a},0}}\word{y}_{\word{a}}^{q_{\word{a}}}\word{x}_{\word{a}}^{p_{\word{a},1}}\word{y}_{\word{a}}^{q_{\word{a}}}
 \cdots \word{y}_{\word{a}}^{q_{\word{a}}}\word{x}_{\word{a}}^{p_{\word{a},k_{\word{a}}}},
\end{align}
where $k_{\word{a}}$ is positive and odd, $q_{\word{a}}\ge1$, the~two endpoint exponents $p_{\word{a},0},p_{\word{a},k_{\word{a}}}$ are positive and odd, and all interior exponents $p_{\word{a},i}$ are positive and even. Then $h$ reflects shuffle squares.
\end{theorem}

\begin{proof}
Suppose that $h(w)$ is a~shuffle square, and fix a~nest-free perfect matching~$M$ of~its positions. For~a~letter~$\word{a}$ occurring $n_{\word{a}}>0$ times in~$w$, erase all symbols except $\word{x}_{\word{a}},\word{y}_{\word{a}}$. Since the~pairs of~symbols are disjoint, the~remaining word is $h(\word{a})^{n_{\word{a}}}$ and inherits a~perfect matching from~$M$.

At~a~boundary between successive copies of~$h(\word{a})$, the~two adjacent $\word{x}_{\word{a}}$-runs merge and have even total length $p_{\word{a},k_{\word{a}}}+p_{\word{a},0}$. Consequently, in~$h(\word{a})^{n_{\word{a}}}$ precisely the~first and last $\word{x}_{\word{a}}$-runs have odd lengths, all other $\word{x}_{\word{a}}$-runs have even lengths, and all $\word{y}_{\word{a}}$-runs have length~$q_{\word{a}}$. Apply Lemma~\ref{lem:morph:run-extraction}. It selects complete edges of~$M$, with exactly one endpoint in~each $\word{y}_{\word{a}}$-run.

Perform this selection for~every letter occurring in~$w$. The~union is still a~subset of~$M$ and therefore contains no nests. Each $\word{y}_{\word{a}}$-run is a~contiguous interval of~the~original word $h(w)$, and each image of~$\word{a}$ contributes exactly~$k_{\word{a}}$ selected positions. Read the~selected positions in~order and relabel $\word{y}_{\word{a}}$ as~$\word{a}$. We obtain a~nest-free perfect matching of~equal letters on
\begin{align*}
 D(w),\qquad D(\word{a})=\word{a}^{k_{\word{a}}}.
\end{align*}
Thus $D(w)$ is a~shuffle square. Lemma~\ref{lem:morph:odd-powers} gives $w\in\SSq$.
\end{proof}

\begin{theorem}\label{thm:morph:three-runs}
Let
\begin{align*}
 h(\word{a})=\word{x}_{\word{a}}^{p_{\word{a}}}\word{y}_{\word{a}}^{q_{\word{a}}}\word{x}_{\word{a}}^{r_{\word{a}}},\qquad p_{\word{a}},q_{\word{a}},r_{\word{a}}\ge1,
\end{align*}
where the~symbols satisfy the~disjointness assumptions of~Theorem~\ref{thm:morph:multirun}. Then $h$ reflects shuffle squares if and only if for~every $\word{a}\in A$ at~least one of~$p_{\word{a}},q_{\word{a}},r_{\word{a}}$ is odd.
\end{theorem}

\begin{proof}
If all three exponents for~some~$\word{a}$ are even, divide each of~the~three runs of~$h(\word{a})$ equally between two copies. Then $h(\word{a})$ is a~shuffle square although~$\word{a}$ is not, so reflection fails.

Conversely, assume the~stated parity condition and fix a~nest-free perfect matching~$M$ associated with~a~decomposition of~$h(w)$. If $q_{\word{a}}$ is odd, use the~private symbol~$\word{y}_{\word{a}}$ and the~rank selection in~the~proof of~Proposition~\ref{prop:morph:private-odd}. If $q_{\word{a}}$ is even but $p_{\word{a}}+r_{\word{a}}$ is odd, use~$\word{x}_{\word{a}}$ instead. In~each case the~selection retains complete edges of~$M$ and exactly one position from~each image of~$\word{a}$.

In~the~remaining case, $q_{\word{a}}$ and $p_{\word{a}}+r_{\word{a}}$ are even. The~hypothesis forces both $p_{\word{a}}$ and~$r_{\word{a}}$ to~be odd. Projecting onto~$\{\word{x}_{\word{a}},\word{y}_{\word{a}}\}$ and applying Lemma~\ref{lem:morph:run-extraction}, as~in~the~proof of~Theorem~\ref{thm:morph:multirun}, again selects complete edges of~$M$ with exactly one position from~each image of~$\word{a}$.

Combining these selections over all source letters produces a~submatching of~$M$, with exactly one position from~each image block. Relabel each selected symbol by~its unique source letter. The~result is a~nest-free perfect matching of~equal letters on~$w$, proving that $w\in\SSq$.
\end{proof}

\begin{example}\label{ex:morph:three-runs}
In~particular, $\word{a}\mapsto\word{1221}$ and $\word{b}\mapsto\word{3443}$ define a~morphism reflecting shuffle squares, although all four target symbols have even multiplicity in~their respective images. The~same conclusion holds for~$\word{a}\mapsto\word{1}^3\word{2}^4\word{1}^5$ and $\word{b}\mapsto\word{3}\word{4}^6\word{3}^3$. The~lengths of~the~images, and the~two endpoint exponents within an~image, need not coincide.
\end{example}

\begin{example}\label{ex:morph:multirun}
Consider the~images
\begin{align*}
 H_{\word{a}}=\word{1}\word{2}^{2}\word{1}^{4}\word{2}^{2}
       \word{1}^{2}\word{2}^{2}\word{1}^{3},\qquad
 H_{\word{b}}=\word{3}\word{4}^{2}\word{3}^{2}\word{4}^{2}
       \word{3}^{4}\word{4}^{2}\word{3}.
\end{align*}
They have lengths $16$ and~$14$ and satisfy Theorem~\ref{thm:morph:multirun} with $k_{\word{a}}=k_{\word{b}}=3$ and $q_{\word{a}}=q_{\word{b}}=2$. Hence the~word
\begin{align*}
 W=H_{\word{a}}H_{\word{b}}H_{\word{b}}H_{\word{a}}
\end{align*}
of~length~$60$ is not a~shuffle square: otherwise $\word{abba}$ would be one. Nevertheless, the~multiplicities of~$\word{1},\word{2},\word{3},\word{4}$ in~$W$ are respectively $20,12,16,12$, and the~projections onto~$\{\word{1},\word{2}\}$ and $\{\word{3},\word{4}\}$ are the~ordinary squares $H_{\word{a}}^2$ and~$H_{\word{b}}^2$. The~obstruction is therefore detected by~the~order of~the~different image blocks, rather than by~either projection alone.
\end{example}

\begin{remark}\label{rem:morph:attribution}
For~a~single source letter, the~non-square assertion for~odd powers of~the~images in~Theorem~\ref{thm:morph:multirun} is already covered by~\cite[Corollary~3]{GrytczukPawlikRucinskiPreprint}. The~reflection statements above additionally handle arbitrary arrangements of~images of~different letters. The~extraction argument uses one fixed matching throughout, so the~separate selections remain compatible in~the~original word. None of~the~proofs in~this section requires a~finite computational check.
\end{remark}

\section{Local repairs and largest twins in~doubly binary words}
\label{sec:repairs}

\noindent
The~finite reduction in~Theorem~\ref{thm:C1-reduction} also controls how far a~doubly binary word can be from a~shuffle square. We first allow arbitrary deletions, without assuming an~even Parikh vector, and then locate single repairs close to~an~endpoint. Recall that $\cC^1$ consists precisely of~the~binary words avoiding $\word{111}$ and~$\word{222}$.

\subsection{Deletion distance and largest twins}

For~this section put $\SSq_{\,\varepsilon}=\SSq\cup\{\eps\}$, and define
\begin{align*}
 \delta(w)=\min\bigl\{|D|:D\subseteq\{1,\ldots,|w|\},\
 w|_{\{1,\ldots,|w|\}\setminus D}\in\SSq_{\,\varepsilon}\bigr\}.
\end{align*}
Thus the~retained positions need not be consecutive. A~pair of~\emph{twins} in~a~word~$w$ consists of~two identical subwords supported on~disjoint sets of~positions. The~positions in~each subword need not be consecutive, but are read in~their original order. The~twins are called \emph{perfect} if their supports together cover all positions of~$w$. Let $t(w)$ denote the~maximum common length of~a~pair of~twins in~$w$, with $t(\eps)=0$. Directly from~these definitions,
\begin{align}\label{repair:eq:twins-distance}
  2t(w)+\delta(w)=|w|.
\end{align}

The~parameter $t(w)$ is the~largest-twins parameter studied by~Axenovich, Person, and Puzynina \cite{AxenovichPersonPuzynina}. They proved that every binary word of~length~$n$ has twins covering $n-o(n)$ positions. Deletion distance to~shuffle squares is also studied in~\cite[Sections~5 and~6.2]{GrytczukPawlikRucinskiPreprint}. Theorem~\ref{repair:thm:distance} gives an~exact formula for~this distance throughout~$\cC^1$.

\begin{lemma}\label{repair:lem:residual-deletion}
If a~residual word~$q$ is reachable after reading~$w$ in~\eqref{eq:residual-recurrence2}, then $\delta(w)\leqslant|q|$.
\end{lemma}

\begin{proof}
The~corresponding partial coloring partitions~$w$ into subwords $v$ and~$vq$, in~one order or the~other. Delete the~positions supporting the~final copy of~$q$. The~two remaining subwords both equal~$v$.
\end{proof}

Put
\begin{align}\label{repair:eq:odd-exceptions}
 \mathfrak{u}=\word{12211211221},\qquad
 \overline{\mathfrak{u}}=\word{21122122112}.
\end{align}
We need two finite checks in~addition to~the~periodic-ray description already proved. Here a~\emph{minimal shuffle-square prefix} means a~$C^1$ shuffle square with no proper nonempty shuffle-square prefix. Its minimality concerns prefixes, not arbitrary subwords.

\begin{lemma}\label{repair:lem:finite-deletion}
The~following assertions hold.
\begin{enumerate}[(i)]
\item Among the~$6609$ words in $\bigcup_{n=0}^{38}\cB_n^{(1)}$, every nonempty word of~even length has deletion distance~$2$. Every word of~odd length has deletion distance~$1$, except $\mathfrak{u}$ and~$\overline{\mathfrak{u}}$ in~\eqref{repair:eq:odd-exceptions}, whose distance is~$3$.
\item There are $2980$ minimal shuffle-square prefixes of~length at~most~$38$. Among their $5768$ concatenations $Tr\in\cC^1$ with $r\in\{\mathfrak{u},\overline{\mathfrak{u}}\}$, every word~$Tr$ has deletion distance~$1$.
\end{enumerate}
\end{lemma}

\begin{proof}
Generate the~sets $\cB_n^{(1)}$ by~\eqref{eq:Ck-bad-tree}. Their level sizes in~\eqref{eq:C1-bad-counts} sum to~$6609$ through level~$38$. For~each nonempty even-length word, enumerate pairs of~deleted positions and apply Proposition~\ref{prop:residual-test} to~the~remaining word. At~odd lengths enumerate single deleted positions. The~only failures are the~two words in~\eqref{repair:eq:odd-exceptions}. Deleting positions $1,2,11$ from~$\mathfrak{u}$ leaves $\word{21121122}$, with two copies of~$\word{2112}$ on supports $(1,2,3,7)$ and $(4,5,6,8)$. Complementation handles~$\overline{\mathfrak{u}}$. Parity and the~definition of~$\cB_n^{(1)}$ give all lower bounds in~(i).

For~(ii), the~complete list consists of~the~words
\begin{align*}
 T=v\word{a}\in\cC^1\cap\SSq,
 \qquad v\in\cB_n^{(1)},\quad 0\leqslant n\leqslant37,
 \quad \word{a}\in\cA.
\end{align*}
This gives $2980$ words. Test $Tr\in\cC^1$ for~each of~the~two choices of~$r$, retaining $5768$ pairs, and enumerate one deleted position in~each retained concatenation. Every test succeeds. These finite enumerations, including all negative single-deletion tests for~$u$ and~$\overline u$, are reproduced by~\texttt{verify\_local\_repairs.py}.
\end{proof}

Listing~\ref{lst:repair-deletion} records the~finite checks in~Lemma~\ref{repair:lem:finite-deletion}. The~predicate \texttt{SQ0} extends the~recognizer from~Listing~\ref{lst:residual-test} by~accepting the~empty word.

\begin{lstlisting}[style=pseudocodestyle,float=htbp,
caption={Finite deletion and absorption checks},
label={lst:repair-deletion}]
function SQ0(w)
  return (w = epsilon or SHUFFLE-SQUARE(w))

function DELETABLE(w, d)
  for each d-element subset I of {1,...,length(w)} do
    v <- w with the positions in I deleted
    if SQ0(v) then return true
  return false

B <- BAD-PREFIX-TREE(1, 38)
assert sum(size(B[n]) : n = 0,...,38) = 6609
E <- empty set
for w in union(B[1],...,B[38]) do
  if length(w) is even then
    assert DELETABLE(w, 2)
  else if not DELETABLE(w, 1) then
    insert w into E
assert E = {(*@$\word{12211211221}$@*), (*@$\word{21122122112}$@*)}
for r in E do
  assert DELETABLE(r, 3)

M <- {v*(*@$\word{a}$@*) : v in union(B[0],...,B[37]), (*@$\word{a}$@*) in {(*@$\word{1}$@*),(*@$\word{2}$@*)},
             v*(*@$\word{a}$@*) is in C^1 and SHUFFLE-SQUARE(v*(*@$\word{a}$@*))}
assert size(M) = 2980
P <- {(T,r) : T in M, r in E, T*r is in C^1}
assert size(P) = 5768
for (T,r) in P do
  assert DELETABLE(T*r, 1)
\end{lstlisting}

\begin{theorem}\label{repair:thm:distance}
For~every $w\in\cC^1$,
\begin{align}\label{repair:eq:distance}
\delta(w)=
\begin{cases}
0,&w\in\SSq_{\,\varepsilon},\\
2,&|w|\text{ is even and }w\notin\SSq_{\,\varepsilon},\\
1,&|w|\text{ is odd and }w\notin\{\mathfrak{u},\overline{\mathfrak{u}}\},\\
3,&w\in\{\mathfrak{u},\overline{\mathfrak{u}}\}.
\end{cases}
\end{align}
Consequently, if $|w|=2m+1$, then $t(w)=m$, except that $t(\mathfrak{u})=t(\overline{\mathfrak{u}})=4$. If $|w|=2m$, then $t(w)=m$ for~$w\in\SSq_{\,\varepsilon}$ and $t(w)=m-1$ otherwise.
\end{theorem}

\begin{proof}
First consider a~word with no nonempty shuffle-square prefix. Lengths at~most~$38$ are covered by~Lemma~\ref{repair:lem:finite-deletion}(i). By~the~proof of~Theorem~\ref{thm:C1-reduction}, every longer such word is a~prefix of~one of~the~$422$ periodic rays. For~the~$211$ rays with period~$\word{1122}$, the~length-$38$ seed~$U$ reaches residual~$\word{12}$. The~letter transitions
\begin{align}\label{repair:eq:short-residual-cycle}
 \word{12}\xrightarrow{\word{1}}\word{2}
 \xrightarrow{\word{1}}\word{21}
 \xrightarrow{\word{2}}\word{1}
 \xrightarrow{\word{2}}\word{12}
\end{align}
show that each subsequent odd-length prefix reaches a~residual of~length~$1$, and each even-length prefix reaches one of~length~$2$. Here the~arrows denote the~letter transitions of~\eqref{eq:residual-recurrence2}. They are not restricted to~run boundaries. The~complementary rays have the~same property. Lemma~\ref{repair:lem:residual-deletion} proves the~required upper bounds. No nonempty word under consideration is a~shuffle square, so length parity gives equality.

For~an~arbitrary word, successively remove shortest nonempty shuffle-square prefixes to~obtain
\begin{align*}
 w=S_1\cdots S_k r,
 \qquad r\in\cB_{|r|}^{(1)}.
\end{align*}
This greedy construction does not require even letter multiplicities. If $r=\eps$, then $w\in\SSq_{\,\varepsilon}$. Otherwise, concatenating the~unchanged shuffle-square blocks with a~repaired remainder gives $\delta(w)\leqslant\delta(r)$. Since each $S_i$ has even length, the~only remaining case is
\begin{align*}
 k\geqslant1,\qquad r\in\{\mathfrak{u},\overline{\mathfrak{u}}\}.
\end{align*}
It suffices to~repair $S_k r$ with one deletion. Write $T=S_k$. This is a~minimal shuffle-square prefix. If $|T|\leqslant38$, apply Lemma~\ref{repair:lem:finite-deletion}(ii).

If $|T|>38$, its first~$38$ letters form a~seed of~one of~the~periodic rays, and the~rest follows that ray until its first exit. Hence, up to~complementation,
\begin{align*}
 T=U(\word{1122})^m e,
 \qquad m\geqslant0,\quad e\in\{\word{12},\word{1121}\},
\end{align*}
with residual~$\word{12}$ reachable after $U(\word{1122})^m$. For~a~word~$z$, let $z^{\setminus j}$ denote deletion of~its $j$th position. Direct application of~\eqref{eq:residual-recurrence2} gives the~four transitions in~Table~\ref{repair:tab:absorption}. Thus deleting the~indicated position from~$er$ completes a~shuffle-square coloring of~$Tr$. Complementation covers the~other~$211$ rays. Concatenation with $S_1\cdots S_{k-1}$ proves $\delta(w)\leqslant1$.

All odd lengths require an~odd number of~deletions, and all even-length non-shuffle-squares require at~least two. The~two exceptional words were checked in~Lemma~\ref{repair:lem:finite-deletion}(i). These observations establish~\eqref{repair:eq:distance}. The~statements about twins follow from~\eqref{repair:eq:twins-distance}.
\end{proof}

\begin{table}[ht]
\centering
\caption{The~four finite transitions absorbing the~odd exceptional remainders. In~each row, $\word{12}\xrightarrow{(er)^{\setminus j}}\eps$}
\label{repair:tab:absorption}
\begin{tabular}{@{}ccc@{}}
\toprule
$e$&$r$&$j$\\
\midrule
$\word{12}$&$u$&$11$\\
$\word{12}$&$\overline u$&$1$\\
$\word{1121}$&$u$&$3$\\
$\word{1121}$&$\overline u$&$1$\\
\bottomrule
\end{tabular}
\end{table}

\subsection{One repair near an~endpoint}

An~adjacent transposition interchanges two consecutive different letters. Deleting or inserting a~pair below always means deleting or inserting two consecutive equal letters. Unlike Theorem~\ref{repair:thm:distance}, the~next result assumes that \emph{both} letter multiplicities are even.

\begin{theorem}\label{repair:thm:local}
Let $w\in\cC^1\setminus\SSq$ have even Parikh vector and be nonempty.
\begin{enumerate}[(i)]
\item An~adjacent transposition within the~last $\min\{10,|w|\}$ positions makes~$w$ a~shuffle square.
\item Deleting an~equal pair within the~last $\min\{11,|w|\}$ positions makes~$w$ a~shuffle square.
\item An~adjacent transposition within the~last $\min\{15,|w|\}$ positions makes~$w$ a~shuffle square that still belongs to~$\cC^1$.
\item Write $w=xy$, with $|y|\leqslant6$. For~some such factorization and some $\word{a}\in\cA$, the~word $x\word{a}\word{a} y$ belongs to~$\SSq\cap\cC^1$.
\end{enumerate}
The~constants $10,11,15,6$ above are all optimal. The~corresponding assertions at~the~left endpoint hold by~reversal.
\end{theorem}

\begin{proof}
Let $\mathcal R$ be the~set of~$230$ nonempty even-Parikh words in~Table~\ref{tab:C1-even-bad}. The~following finite checks suffice. For~each $r=\word{r}_1\cdots \word{r}_n\in\mathcal R$, enumerate the~operations specified in~Table~\ref{repair:tab:finite-local}, and test the~result with Proposition~\ref{prop:residual-test}. In~the~last two rows, also require membership in~$\cC^1$. Each row succeeds for~all~$230$ words. Listing~\ref{lst:repair-local} specifies these checks. They are performed independently by~\texttt{verify\_local\_repairs.py} after regenerating~$\mathcal R$.

By~Corollary~\ref{cor:C1-factorization}, write $w=S_1\cdots S_k r$, where $r\in\mathcal R$. Apply the~appropriate operation inside~$r$. The~resulting remainder is a~shuffle square, and concatenation with the~unchanged $S_i$ gives a~shuffle square. Distances from the~right endpoint are unchanged. In~(iii) and~(iv), the~finite checks leave the~first two letters of~$r$ unchanged. Thus they also preserve the~$C^1$ condition across the~boundary between $S_k$ and~$r$. All other boundaries are untouched.

For~sharpness, direct enumeration of~all eligible operations in~the~words of~Table~\ref{repair:tab:sharp-local} gives the~stated smallest window or suffix length. In~particular, the~successful transpositions for~the~first word start at~positions $1,3,8,10,15$. Among these, exactly $1,3,8,10$ preserve~$C^1$. The~successful pair deletions for~the~second word start at~positions $2,4,7,11,14$. These complete lists prove the~first three lower bounds. The~last lower bound follows by~checking both inserted pairs at~each of~the~last six gaps, and finding no repair in~$\cC^1$ there. A~repair exists before the~suffix of~length~$6$. Reversal preserves shuffle squares and~$\cC^1$ and converts all right-end assertions to~left-end assertions.
\end{proof}

\begin{table}[ht]
\centering
\caption{Finite checks on~each of~the~$230$ terminal remainders. Positions are numbered from~$1$. In~the~last row the~gap follows position~$j$}
\label{repair:tab:finite-local}
\small
\begin{tabular}{@{}lll@{}}
\toprule
Operation&Eligible positions&Additional condition\\
\midrule
Transpose $\word{r}_j,\word{r}_{j+1}$&$\max\{1,n-9\}\leqslant j<n$&$\word{r}_j\ne \word{r}_{j+1}$\\
Delete $\word{r}_j\word{r}_{j+1}$&$\max\{1,n-10\}\leqslant j<n$&$\word{r}_j=\word{r}_{j+1}$\\
Transpose, preserve $C^1$&$\max\{3,n-14\}\leqslant j<n$&$\word{r}_j\ne \word{r}_{j+1}$\\
Insert $\word{11}$ or $\word{22}$, preserve $C^1$&$\max\{2,n-6\}\leqslant j\leqslant n$&---\\
\bottomrule
\end{tabular}
\end{table}

\begin{table}[ht]
\centering
\caption{Examples establishing the~optimal locality constants}
\label{repair:tab:sharp-local}
\small
\begin{tabular}{@{}llr@{}}
\toprule
Operation&Word&Optimal constant\\
\midrule
Transposition&$\word{122112112122122122112112}$&$10$\\
Pair deletion&$\word{122112112122122112211212}$&$11$\\
Transposition preserving $C^1$&$\word{122112112122122122112112}$&$15$\\
Pair insertion preserving $C^1$&$\word{12211211212212212112}$&$6$\\
\bottomrule
\end{tabular}
\end{table}

\begin{example}\label{repair:ex:class-preservation}
The~distinction between (i) and~(iii) is necessary. For~the~first word of~Table~\ref{repair:tab:sharp-local}, the~rightmost successful transposition gives
\begin{align*}
 \word{122112112122121222112112},
\end{align*}
which contains~$\word{222}$. Transposing positions $10,11$ instead gives
\begin{align*}
 \word{122112112212122122112112}\in\SSq\cap\cC^1.
\end{align*}
Nor can one always strengthen (ii) by~requiring the~result to~remain in~$\cC^1$: deleting either available equal pair from~$\word{12211212}$ gives $\word{111212}$ or~$\word{122212}$. The~first is a~shuffle square and the~second is not, but neither is doubly binary.
\end{example}

\subsection{Completion at~the~endpoint and longer runs}

\begin{theorem}\label{repair:thm:completion}
For~every $w\in\cC^1$ with even Parikh vector, there is a~word
\begin{align*}
 z\in\{\word{1122},\word{1212},\word{1221},
          \word{2112},\word{2121},\word{2211}\}
\end{align*}
such that $wz\in\SSq\cap\cC^1$. Four appended letters are sometimes necessary, even if the~completed word need not remain in~$\cC^1$.
\end{theorem}

\begin{proof}
For~each $r\in\mathcal R$, testing the~six displayed words finds at~least one~$z$ with $rz\in\SSq\cap\cC^1$. This finite check is included in~Listing~\ref{lst:repair-local} and in~\texttt{verify\_local\_repairs.py}. Apply it to~the~terminal remainder in~Corollary~\ref{cor:C1-factorization}. No existing letter is changed, so the~boundary with the~preceding shuffle-square blocks is preserved. If the~remainder is empty, append $\word{1122}$ when $w$ is empty or ends in~$\word{2}$, and append $\word{2211}$ otherwise. Each chosen suffix is itself a~shuffle square and creates no run of~length~$3$.

For~optimality, let $w=\word{122112112122}$. The~word~$w$ is not a~shuffle square. Any appended word of~length less than~$4$ that could complete it to~a~shuffle square must, by~parity, be $\eps$, $\word{11}$, or~$\word{22}$. The~last two choices fail the~exact recognizer, whereas
\begin{align*}
 w\word{1122}=\word{1221121121221122}
\end{align*}
is a~shuffle square with root~$\word{12211122}$.
\end{proof}

Listing~\ref{lst:repair-local} records the~finite checks for~Theorems~\ref{repair:thm:local} and~\ref{repair:thm:completion}, including sharpness. Positions are numbered from~$1$, an~insertion gap~$j$ follows the~first~$j$ letters, and slices have inclusive endpoints, with an~empty slice when the~lower endpoint exceeds the~upper one. Each pair~$(j,d)$ records a~successful repair at~position or gap~$j$. For~transpositions and deletions, $d=n-j+1$ is the~length of~the~suffix starting at~the~first affected position; for~insertions, $d=n-j$ is the~length of~the~suffix following the~gap. Thus minimizing~$d$ locates the~successful repair nearest the~right endpoint.

\begin{lstlisting}[style=pseudocodestyle,float=htbp,
caption={Local repairs, sharpness, and balanced completions},
label={lst:repair-local}]
function REPAIRS(r)
  n <- length(r)
  S[swap], S[delete], S[swapC1], S[insertC1] <- empty sets
  for j <- 1,...,n-1 do
    if (*@$\word{r}_{j}$@*) != (*@$\word{r}_{j+1}$@*) then
      v <- r[1:j-1]*(*@$\word{r}_{j+1}$@*)*(*@$\word{r}_{j}$@*)*r[j+2:n]
      if SQ0(v) then
        insert (j,n-j+1) into S[swap]
        if v is in C^1 then
          insert (j,n-j+1) into S[swapC1]
    else
      v <- r[1:j-1]*r[j+2:n]
      if SQ0(v) then insert (j,n-j+1) into S[delete]
  for j <- 0,...,n and (*@$\word{a}$@*) in {(*@$\word{1}$@*),(*@$\word{2}$@*)} do
    v <- r[1:j]*(*@$\word{a}$@*)*(*@$\word{a}$@*)*r[j+1:n]
    if SQ0(v) and v is in C^1 then
      insert (j,n-j) into S[insertC1]
  return S

B <- BAD-PREFIX-TREE(1, 38)
R <- {r in union(B[1],...,B[38]) : r has even Parikh vector}
assert size(R) = 230
for r in R do
  S <- REPAIRS(r)
  for (op,b,j0) in {(swap,10,1),(delete,11,1),
                    (swapC1,15,3),(insertC1,6,2)} do
    assert some (j,d) in S[op] has j >= j0 and d <= b
  assert some z in {(*@$\word{1122}$@*),(*@$\word{1212}$@*),(*@$\word{1221}$@*),(*@$\word{2112}$@*),(*@$\word{2121}$@*),(*@$\word{2211}$@*)}
         satisfies SQ0(r*z) and r*z is in C^1

# Sharpness uses all positions, without the restriction j >= j0.
for (op,r,b) in the rows of Table (*@\ref{repair:tab:sharp-local}@*) do
  S <- REPAIRS(r)
  assert S[op] is nonempty
  assert min{d : (j,d) in S[op]} = b

w <- (*@$\word{122112112122}$@*)
assert not SQ0(w) and not SQ0(w*(*@$\word{11}$@*)) and not SQ0(w*(*@$\word{22}$@*))
assert SQ0(w*(*@$\word{1122}$@*)) and w*(*@$\word{1122}$@*) is in C^1
\end{lstlisting}

For~a~binary word~$w$ with run lengths $\ell_1,\ldots,\ell_s$, put
\begin{align*}
 E(w)=\sum_{i=1}^s\left\lfloor\frac{\ell_i-1}{2}\right\rfloor,
 \qquad E(\eps)=0.
\end{align*}
This quantity counts the~equal-pair deletions needed to~reduce each run to~length~$1$ or~$2$ while retaining its parity.

\begin{corollary}\label{repair:cor:long-runs}
Every binary word~$w$ with even Parikh vector can be reduced to~a~member of~$\SSq_{\,\varepsilon}$ by~at~most $E(w)+1$ deletions of~adjacent equal pairs. In~particular,
\begin{align*}
 \delta(w)\leqslant2E(w)+2.
\end{align*}
The~bound $E(w)+1$ for~the~number of~equal-pair deletions is sharp.
\end{corollary}

\begin{proof}
Reduce every odd run to~length~$1$ and every even run to~length~$2$. Exactly $E(w)$ equal-pair deletions produce a~$C^1$ word with even Parikh vector. If it does not already belong to~$\SSq_{\,\varepsilon}$, one further pair deletion suffices by~Theorem~\ref{repair:thm:local}(ii).

For~sharpness, consider $w=\word{1}\word{2}^{2k}\word{1}$, $k\geqslant1$. Until all middle letters have been removed, the~only deletable equal pairs consist of~two~$\word{2}$'s. No word $\word{1}\word{2}^{2j}\word{1}$ with $j\geqslant1$ is a~shuffle square: its two copies would each contain one~$\word{1}$, while one would begin with that letter and the~other would end with it, which is impossible for~equal roots containing~$\word{2}$. Thus exactly $k=E(w)+1$ pair deletions are necessary. This sharpness concerns adjacent equal pairs. Unrestricted deletions can simply remove the~two endpoint letters.
\end{proof}

\begin{remark}\label{repair:rem:run-three}
The~single-repair conclusion does not extend unchanged to~words whose runs have length at~most~$3$. For~example, direct enumeration gives $\delta(\word{12221112})=4$. Deleting an~equal pair from each of~its two length-$3$ runs leaves $\word{1212}$. Thus at~least two equal-pair deletions are needed in~this larger class. It is natural to~ask whether two such deletions suffice for~every even-Parikh word with run lengths at~most~$3$. The~results above do not settle this question.
\end{remark}

\section{Scope of~the~computer assistance}
\label{sec:computer-assistance}

\noindent
The~finite-differentiability and Kolakoski theorems, and the~local-repair results of~Section~\ref{sec:repairs}, use explicitly bounded computations. The~proofs specify why these finite tasks suffice for~arbitrary lengths: strong induction for~$C^3$, terminating extension layers for~$C^2$, and residual loops and exits for~the~$422$ $C^1$ rays. In~Theorem~\ref{repair:thm:distance}, the~short residual cycle and the~four absorption transitions additionally handle every phase of~the~rays and every long minimal shuffle-square prefix. No unbounded assertion is inferred solely from~a~cutoff or stabilization of~counts. The~morphism results in~Section~\ref{sec:morphisms} have purely mathematical proofs.

The~supplementary material contains complete Python implementations of~the~finite proof computations, a~shared recognition module, and a~driver script. The~programs use only the~Python standard library and regenerate their input sets without external data. Their roles are as~follows.
\begin{enumerate}[(i)]
\item \classcode{} verifies the~$C^1$, $C^2$, and~$C^3$ computations. For~$C^3$, these include the~$175$-node tree, its $46\,600$ parity-base words, and $34$ exceptions. It also checks that all~$34$ exceptions and the~terminal words of~length~$23$ are smooth. For~$C^2$, the~program checks the~$823$-node tree, $208$ short exceptions, $288$ minimal suffix squares, and extension layers of~$13\,752$ and~$200$ candidates, leaving four and zero exceptions, respectively. For~$C^1$, it checks the~$230$ even-Parikh cores and all~$422$ periodic seeds, including the~residual loops and exits. Three implementations of~the~exact recognizers agree on~every nonempty binary word through length~$16$.

\item \repaircode{} regenerates the~$6\,609$ bad-prefix words through length~$38$, the~$2\,980$ minimal shuffle-square prefixes, and all~$5\,768$ admissible concatenations with the~two odd exceptions. It checks deletion bounds, periodic residuals, absorption transitions, all four sharp locality constants, boundary preservation, and balanced completions, including the~finite checks establishing sharpness. These are the~finite tasks used in~Section~\ref{sec:repairs}. The~finite deletion and local-modification searches are specified in~Listings~\ref{lst:repair-deletion} and~\ref{lst:repair-local}.

\item \texttt{shuffle\_core.py} supplies the~exact recognition and residual-transition routines used by~\repaircode{}. As~an~implementation check, it compares its run and letter recurrences with independent literal support enumeration on~all~$8\,191$ binary words through length~$12$, including the~empty word.

\item \texttt{verify\_all.py} runs the~three preceding programs in~fresh subprocesses, stops on~failure, and records their exit statuses and execution times.
\end{enumerate}

The~smooth-word conclusions follow from~the~$C^3$ theorem and the~smoothness checks above, as~explained in~Corollary~\ref{cor:stable-regime}. The~Kolakoski characterization follows from~the~same theorem and the~explicit short-word arguments in~Section~\ref{sec:kolakoski-prefixes}. No additional enumeration of~smooth words or bounded testing of~morphisms is used in~these proofs. The~Kolakoski arguments use no assumption about letter or factor frequencies.

After extracting the~supplementary archive, run \texttt{python verify\_all.py} from the~directory containing the~driver script. This command runs all checks afresh and writes computed JSON data and execution logs. The~complete verification succeeded under CPython~3.12.14. The~entry point rejects the~assertion-disabling options \texttt{-O} and~\texttt{-OO}. The~file \texttt{README.md} gives requirements, commands, a~program-to-result mapping, and expected results.

Computed data and execution logs accompany the~programs as~reference outputs; they are not inputs to~the~verification. The~article contains the~recurrences, pseudocode, decisive counts, and mathematical arguments explaining why the~finite verifications establish the~stated results.

\section*{Statements and Declarations}

\noindent\textbf{Funding.} The~author received no financial support for~the~research, authorship, or publication of~this article.

\noindent\textbf{Competing interests.} The~author declares that he has no competing interests.

\noindent\textbf{Data availability.} The~Python programs and computed outputs supporting the~finite proof computations are included in~the~supplementary material described in~Section~\ref{sec:computer-assistance}. All input sets are generated by~the~programs; no external datasets are required.

\noindent\textbf{Use of generative artificial intelligence.} During the~preparation of~this manuscript, the~author used OpenAI's ChatGPT (GPT-6) as~an~auxiliary tool to~assist with writing the~Python verification scripts,  literature search, and language editing. All AI-assisted content was independently reviewed and verified by~the~author, who takes full responsibility for~the~manuscript.

\end{document}